\documentclass[10pt,reqno]{amsart}
\usepackage{amsmath,amscd,amsfonts,amsthm,amssymb,latexsym}

\theoremstyle{plain}
   \newtheorem{theorem}{Theorem}[section]
   \newtheorem{proposition}[theorem]{Proposition}
   \newtheorem{lemma}[theorem]{Lemma}
   \newtheorem{corollary}[theorem]{Corollary}
   \newtheorem{problem}{Problem}

\theoremstyle{definition}
   \newtheorem{definition}{Definition}[section]
   
   \newtheorem{notation}{Notation}[section]
\theoremstyle{remark}
   \newtheorem{remark}{Remark}[section]

\author[J.~Borcea]{Julius Borcea}
\address{Department of Mathematics, Stockholm University, SE-106 91 Stockholm,
Sweden}
\email{julius@math.su.se}

\author[P.~Br\"and\'en]{Petter Br\"and\'en}
\address{Department of Mathematics, Royal Institute of Technology, 
SE-100 44 Stockholm, Sweden}
\email{pbranden@math.kth.se}
       

\thanks{The first author is supported by the Royal Swedish Academy 
of Sciences. The second author is 
supported by the G\"oran Gustafsson Foundation.}

\numberwithin{equation}{section}

\newcommand{\NN}{\mathbb{N}}

\newcommand{\bH}{\mathbb{H}}

\newcommand{\NV}{\mathcal{N}}
\newcommand{\LY}{\mathcal{LY}}

\newcommand{\HH}{\mathcal{H}}
\newcommand{\MM}{\mathcal{M}}
\newcommand{\HHH}{\overline{\mathcal{H}}}

\newcommand{\RR}{\mathbb{R}}
\newcommand{\CC}{\mathbb{C}}
\newcommand{\DD}{\mathbb{D}}
\newcommand{\KK}{\mathbb{K}}
\newcommand{\CM}{\mathbb{C}_{ {MA}}}

\newcommand{\te}{\theta}
\newcommand{\al}{\alpha}

\renewcommand{\Im}{{\rm Im}}
\renewcommand{\Re}{{\rm Re}}

\def\newop#1{\expandafter\def\csname #1\endcsname{\mathop{\rm
#1}\nolimits}}

\newop{diag}
\newop{supp}
\newop{Proj}
\newop{JP}
\newop{Hom}
\newop{Sym}
\newop{Symb}
\newop{PolP}
\newop{PolO}
\newop{glb}
\newop{MAP}
\newop{MOD}

\begin{document}

\title[The Lee-Yang and P\'olya-Schur Programs. I.]{The Lee-Yang 
and P\'olya-Schur Programs. I. \\ Linear Operators Preserving Stability}

\begin{abstract} 
In 1952 Lee and Yang proposed the program of analyzing phase 
transitions in terms of zeros of 
partition functions. Linear 
operators preserving non-vanishing properties are essential in this 
program and various contexts in complex analysis, 
probability theory, combinatorics, and matrix theory. 
We characterize all linear operators on finite or infinite-dimensional
spaces of multivariate polynomials preserving 
the property of being non-vanishing whenever the variables are in 
prescribed open circular domains. In particular, this 
solves the higher dimensional counterpart of 
a long-standing classification problem originating from classical works of 
Hermite, Laguerre, Hurwitz and P\'olya-Schur on univariate polynomials with 
such properties. 
\end{abstract}

\maketitle

\tableofcontents

\section*{Introduction}\label{s-1}

Zero loci of multivariate polynomials or transcendental entire functions and 
their dynamics under linear transformations are central topics in geometric 
function theory which in recent years has found applications in 
statistical mechanics, combinatorics, probability theory and 
matrix theory \cite{BB-II,BBS2,BBCV,BBL,Br1,COSW,sok}. For instance, as shown 
in \cite{BB-II,COSW,LS}, the Lee-Yang program on phase 
transitions in equilibrium statistical mechanics \cite{LY,LY1}
is intimately related to the problems of 
describing linear operators on polynomial spaces that preserve the property 
of being non-vanishing when the variables are in prescribed 
subsets of $\CC^n$. 
In this paper we fully solve these problems 
for (products of) open circular domains. In particular, this 
accomplishes a long-standing classification program 
originating from classical works of Hermite, Laguerre, Hurwitz and 
P\'olya-Schur on univariate polynomials with such properties 
\cite{CC1,cso,lag,PS,PSz,RS}. 

Given an integer $n\ge 1$ and $\Omega \subset \CC^n$ we say 
that $f\in\CC[z_1,\ldots,z_n]$ is 
$\Omega$-{\em stable}
if $f(z_1,\ldots,z_n)\neq 0$ whenever $(z_1,\ldots,z_n)\in\Omega$. 
A $\KK$-linear operator 
$T:V\to \KK[z_1,\ldots,z_n]$, where $\KK=\RR$ or $\CC$ and 
$V$ is a subspace of $\KK[z_1,\ldots,z_n]$, is said to {\em preserve} 
$\Omega$-{\em stability} if for any 
$\Omega$-stable polynomial $f\in V$ the 
polynomial $T(f)$ is either 
$\Omega$-stable or $T(f)\equiv 0$.
For $\kappa=(\kappa_1,\ldots,\kappa_n) \in \NN^n$ let 
$\KK_\kappa[z_1,\ldots, z_n]
=\{f\in \KK[z_1,\ldots, z_n]:\deg_{z_i}(f)\le \kappa_i,1\le i\le n\}$, 
where $\deg_{z_i}(f)$ is the degree of $f$ in $z_i$. If $\Psi\subset \CC$ and 
$\Omega=\Psi^n$ then $\Omega$-stable polynomials will also be referred to as 
$\Psi$-stable. The fundamental classification problems that we address below  
are as follows.

\begin{problem}\label{prob1}
Characterize all linear operators 
$T:\KK_\kappa[z_1,\ldots, z_n]\to \KK[z_1,\ldots,z_n]$ that preserve 
$\Omega$-stability for a given set $\Omega\subset\CC^n$ and 
$\kappa\in\NN^n$.
\end{problem}

\begin{problem}\label{prob2}
Characterize all linear operators 
$T:\KK[z_1,\ldots, z_n]\to \KK[z_1,\ldots,z_n]$ that preserve 
$\Omega$-stability, where $\Omega$ is a prescribed subset of 
$\CC^n$.
\end{problem}

For $n=1$ these problems were stated in this general form in 
\cite{CC1,cso}, see also \cite{iserles1} and \cite[pp.~182-183]{RS}. Their
natural multivariate analogs given above were formulated explicitly in 
\cite{BBS3,BBCV}, thereby encompassing essentially all similar questions or 
variations thereof that have been much studied for more than a century.
Note that for $n=1$, $\KK=\RR$ and $\Omega=\{z\in\CC:\Im(z)>0\}$ 
Problems \ref{prob1}--\ref{prob2} amount to classifying linear operators that 
preserve the set of real polynomials with all real zeros. This question has 
a long and distinguished history going back to Laguerre and 
P\'olya-Schur, see \cite{BBCV,CC1} and the references 
therein. In particular, in \cite{PS} P\'olya and Schur characterized all 
diagonal operators with this property (so-called multiplier sequences), 
which paved the way for numerous subsequent investigations 
\cite{BBS1,BBS3,CCS,CC2,cso,fisk,iserles1,Le,M,pol-coll,RS,schur}.
However, it was not until very recently that full solutions to this question 
-- and, more generally, to Problems \ref{prob1}--\ref{prob2} for $n=1$ and any 
open circular domain $\Omega$ -- were obtained in \cite{BBS1}.  

In physics literature (\cite{sco-sok}, \cite[\S 5]{sok}) one sometimes 
distinguishes between ``soft'' and ``hard'' theorems asserting the 
non-vanishing of partition functions in certain regions. This dichotomy stems 
from ``soft-core'' vs.~``hard-core'' pair interactions in lattice-gas models 
and does not refer to the level of difficulty in proving such theorems but 
rather to the fact that in some sense ``soft'' theorems are constraint-free 
while ``hard'' theorems involve constraints of various kinds, such as the 
maximum degree of a graph. By analogy with this terminology, one may say 
that results pertaining to Problem \ref{prob1} are ``hard'' or ``algebraic'' 
(bounded degree) while those for Problem \ref{prob2} are ``soft'' or 
``transcendental'' (unbounded degree). For $n\ge 1$ several partial 
results of the latter type were established and appplied to various questions 
in e.g.~combinatorics and statistical mechanics in the past two decades. 
In \cite{LS} Lieb and Sokal proved a 
general result \cite[Proposition~2.2]{LS} pertaining to Problem~\ref{prob2} 
and more results of the same kind were recently obtained in \cite{COSW,W2}. 
Another contribution to this subject is Hinkkanen's Schur-Hadamard 
composition theorem \cite{hink}. Progress on Problem~\ref{prob2}  
was recently made in \cite{BBS3} where a complete 
solution was given for finite order partial differential operators 
when $\Omega=\Psi^n$ and $\Psi$ is an open half-plane. Further 
contributions to this problem have been reported in \cite{fisk}.

The main results of this paper provide complete solutions to 
Problems~\ref{prob1}--\ref{prob2} when 
$\Omega=\Omega_1\times\cdots\times\Omega_n$ and $\Omega_i$, $1\le i\le n$, 
are arbitrary open circular domains in $\CC$. We also characterize all 
linear transformations preserving the class of {\em Lee-Yang polynomials} 
defined with respect to any such set $\Omega$. In particular, the 
classification 
theorems of \cite{BBS1} may now be viewed as special cases of the theorems 
below from which they follow by setting $n=1$. To achieve this we need 
several new ideas and results. For instance, we define the 
{\em polarization} of a linear operator to reduce the sufficiency 
part for arbitrary degrees in the algebraic characterization to the case of 
linear transformations acting on polynomials of degree at most one in each 
variable. For the necessity part in the transcendental characterization 
we generalize Sz\'asz' inequalities \cite{szasz} to several variables. 
These are bounds on the 
coefficients of a stable polynomial that only depend on its first few 
non-zero coefficients. 

The solutions to Problems \ref{prob1}--\ref{prob2} 
may be summarized as follows: (essentially all) linear operators 
preserving $n$-variate stable/Lee-Yang polynomials are induced by 
$2n$-variate stable/Lee-Yang 
polynomials via appropriately defined symbol maps. 
 
In a sequel \cite{BB-II} to this paper 
we build on our classification theorems 
to develop a self-contained theory of 
multivariate stable polynomials. We therefore take care to rely on as 
few auxiliary results as possible in the process. 
Combined with the 
present work, the theory and applications in \cite{BB-II} also yield 
a natural framework for dealing in a uniform manner with 
Lee-Yang type problems in statistical mechanics, combinatorics, and geometric
function theory in one or several variables, thus contributing to  
Ruelle's quest \cite{ruelle2} for an appropriate mathematical 
context encompassing the celebrated Lee-Yang theorem \cite{LY} 
(locating the zeros of 
the partition function of the ferromagnetic Ising model on the imaginary axis 
in the complex fugacity plane) and its many modern relatives 
\cite{As,beau,BBCK,BBCKK,HL,KOS,LS,new1,new2,ruelle1,ruelle4,sco-sok}. 
This also illustrates Hinkkanen's observations about zeros of multivariate 
polynomials and 
their ``so far unnoticed connections to various other concepts in 
mathematics'' \cite{hink}, which are further substantiated by recent 
applications to probability theory, matrix theory and combinatorics 
\cite{BBS2,BBL,Br1,COSW,Lig1,ruelle-g2,sok2,W2}.

\section{Operator Symbols and $\Omega$-Stable Polynomials}\label{new-s-prel}

The general notion of $\Omega$-stability defined in the introduction extends 
classical univariate concepts such as Hurwitz or continuous-time 
stable polynomials and Schur or discrete-time stable polynomials 
(see, e.g., \cite{M,RS}). These correspond to $n=1$, 
$\Omega=\{z \in \CC : \Re(z) \ge 0\}$ and $n=1$, 
$\Omega=\{z \in \CC : |z|\ge 1\}$, respectively.

Let now $n$ be an arbitrary positive integer. For $\te\in [0,2\pi)$ set
$$ 
\bH_\te=\{z \in \CC: \Im(e^{i\te}z)>0\}.
$$
We find it convenient to work with the upper half-plane 
$$
H:=\bH_0=\{z \in \CC: \Im(z)>0\},
$$
therefore we will follow Levin's terminology \cite{Le} and refer to 
$H$-stable multivariate polynomials simply 
as {\em stable} polynomials (cf.~\cite{BBS1,BBS2,BBS3,BBL}). 
A stable polynomial with all real coefficients is called {\em real stable}. 
Clearly, a univariate 
real polynomial is stable if and only if all its zeros are real. 
We denote the sets of stable, respectively
real stable polynomials in $n$ variables by $\HH_n(\CC)$ and $\HH_n(\RR)$,
respectively.

If $\Omega = \bH_{\frac{\pi}{2}}=\{z \in \CC : \Re(z) >0\}$ 
then $\Omega$-stable polynomials are called {\em weakly Hurwitz stable}. In 
\cite{COSW} these are termed polynomials with the {\em half-plane 
property}. 

\subsection{Algebraic and Transcendental Symbols}\label{new-ss-not}

The solutions to Problems \ref{prob1}--\ref{prob2} for circular domains make 
use of appropriately defined operator symbols that we proceed to describe. 
For simplicity, in this section we only state the main theorems for linear 
operators preserving ($H$-)stability. 

Recall that $\KK_\kappa[z_1,\ldots, z_n]$, where $\KK=\RR$ or $\CC$ and 
$\kappa=(\kappa_1,\ldots,\kappa_n) \in \NN^n$, 
is the $\KK$-space of polynomials in variables $z_1,\ldots,z_n$ of 
degree at most $\kappa_i$ in $z_i$, $1\le i\le n$. If  
$\gamma=(\gamma_1,\ldots,\gamma_n)\in\NN^n$ we define the 
{\em algebraic symbol} of a linear operator 
$$
T: \KK_\kappa[z_1,\ldots, z_n] \rightarrow  \KK_\gamma [z_1,\ldots, z_n]$$
to be the polynomial 
$G_T(z,w) \in \KK_{\gamma \oplus \kappa}[z_1,\ldots, z_n, w_1,\ldots, w_n]$ 
given by 
$$
G_T(z,w)= T[(z+w)^\kappa]
= \sum_{\alpha \leq \kappa}\binom \kappa \alpha T(z^\alpha)w^{\kappa-\alpha}, 
$$
where $\gamma \oplus \kappa = (\gamma_1,\ldots, \gamma_n, \kappa_1, \ldots, 
\kappa_n)$ while $\leq$ denotes the standard product partial order on $\NN^n$ 
and we employ the usual multi-index notations $z=(z_1,\ldots, z_n)$, 
$w=(w_1,\ldots, w_n)$, $z^{\al}=\prod_{i=1}^{n}z_i^{\al_i}$, 
$zw=(z_1w_1,\ldots,z_nw_n)$ and 
$$
\binom \kappa \alpha = \begin{cases}
\prod_{j=1}^n \frac{\kappa_j!}{\alpha_j!(\kappa_j-\alpha_j)!} \mbox{ if } 
\alpha \leq \kappa, \\
0 \mbox{ otherwise}. 
\end{cases}
$$

Our first theorem provides an algebraic characterization of stability 
preservers on finite-dimensional complex polynomial spaces, which solves 
Problem \ref{prob1} for $\KK=\CC$ and $\Omega=H$.

\begin{theorem}\label{multi-finite-stab}
Let $\kappa \in \NN^n$ and $T : \CC_\kappa[z_1,\ldots, z_n] \rightarrow 
\CC[z_1,\ldots, z_n]$ be a linear operator. Then $T$ preserves stability 
if and only if either
\begin{itemize}
\item[(a)] $T$ has range of dimension at most one and is of the form 
$$
T(f) = \alpha(f)P,
$$
where $\alpha$ is a linear functional on $\CC_\kappa[z_1,\ldots, z_n]$ and 
$P$ is a stable polynomial, or 
\item[(b)] $G_T(z,w)\in \HH_{2n}(\CC)$. 
\end{itemize}
\end{theorem}

\begin{remark}\label{def-symb}
Theorem \ref{multi-finite-stab} may also be stated in terms of other
(algebraic) symbols of $T$ that are equivalent to the one defined above,
for instance the polynomial
$$
T[(1-zw)^\kappa]
= \sum_{\alpha \leq \kappa}(-1)^{\alpha}\binom \kappa \alpha 
T(z^\alpha)w^{\alpha}=(-1)^{\kappa}w^{\kappa}G_T(z,-w^{-1}).
$$
Note that $T[(1-zw)^\kappa]\in\HH_{2n}(\CC)\Leftrightarrow 
G_T(z,w)\in \HH_{2n}(\CC)$, see, e.g., Lemma \ref{l-close} (3).
\end{remark}

\begin{definition}\label{d-pp}
Two polynomials $f,g \in \RR[z_1,\ldots,z_n]$ are in 
{\em proper position}, denoted $f \ll g$, if $g+if \in \HH_n(\CC)$.
\end{definition} 

By the Hermite-Biehler theorem (see, e.g., \cite[Theorem 6.3.4]{RS}) 
in the univariate case the relation $f\ll g$ is equivalent to saying that
$f$ and $g$ are real-rooted (or identically zero), their zeros interlace
and their {\em Wronskian} $W[f,g]:=f'g-fg'$ is non-positive on the whole of 
$\RR$. We study further properties of this relation in \S \ref{s-2}.

The algebraic characterization of real stability preservers on 
finite-dimensional polynomial spaces -- which solves 
Problem \ref{prob1} for $\KK=\RR$ and $\Omega=H$ -- is as follows.

\begin{theorem}\label{multi-finite-hyp}
Let $\kappa \in \NN^n$ and 
$T : \RR_\kappa[z_1,\ldots, z_n] \rightarrow \RR[z_1,\ldots, z_n]$ be a 
linear operator. Then $T$ preserves real stability if and only if either
\begin{itemize}
\item[(a)] $T$ has range of dimension no greater than two and is of the form 
$$
T(f) = \alpha(f)P + \beta(f)Q,
$$
where $\alpha, \beta: \RR_\kappa[z_1,\ldots, z_n] \rightarrow \RR$ are 
linear functionals and $P,Q$ are real stable polynomials such that 
$P \ll Q$, or 
\item[(b)] $G_T(z,w)\in \HH_{2n}(\RR)$, or 
\item[(c)] $G_T(z,-w)\in \HH_{2n}(\RR)$. 
\end{itemize}
\end{theorem}

If $T : \KK[z_1,\ldots, z_n] \rightarrow \KK[z_1,\ldots,z_n]$, where 
$\KK=\RR$ or 
$\CC$, is a linear operator 
we define its {\em transcendental symbol}, $\overline{G}_T(z,w)$, to be the 
formal power series in $w_1, \ldots, w_n$ with polynomial coefficients in 
$\KK[z_1,\ldots,z_n]$ given by 
$$
\overline{G}_T(z,w)
=\sum_{\alpha \in \NN^n} (-1)^\alpha T(z^\alpha)
\frac {w^\alpha}{\alpha!}, 
$$
where  
$\alpha!= \alpha_1! \cdots \alpha_n!$. By abuse of notation we write
$\overline{G}_T(z,w) =T[e^{-z\cdot w}]$, where 
$z\cdot w=z_1w_1+\ldots+z_nw_n$.

Let us define the {\em complex Laguerre-P\'olya class} $\HHH_n(\CC)$ as 
the class of entire functions in $n$ variables that are limits, uniformly on 
compact sets, of polynomials in $\HH_n(\CC)$, see, e.g., \cite[Chap.~IX]{Le}. 
The usual {\em Laguerre-P\'olya class} $\HHH_n(\RR)$ consists of all 
functions in $\HHH_n(\CC)$ with real coefficients.

Our next theorem provides a transcendental characterization of stability
preservers on infinite-dimensional complex polynomial spaces, thus solving 
Problem \ref{prob2} for $\KK=\CC$ and $\Omega=H$.

\begin{theorem}\label{multi-infinite-stab}
Let $T : \CC[z_1,\ldots, z_n] \rightarrow \CC[z_1,\ldots, z_n]$ be a 
linear operator. Then $T$ preserves stability if and only if either
\begin{itemize}
\item[(a)] $T$ has range of dimension at most one and is of the form 
$$
T(f) = \alpha(f)P,
$$
where $\alpha$ is a linear functional on $\CC[z_1,\ldots, z_n]$ and $P$ is a 
stable polynomial, or 
\item[(b)] $\overline{G}_T(z,w)\in \HHH_{2n}(\CC)$. 
\end{itemize}
\end{theorem}

\begin{remark}\label{lhp}
From Theorem~\ref{multi-infinite-stab} one can easily deduce a 
characterization of linear operators preserving $\Omega$-stability for any 
open half-plane $\Omega$. For instance, the analog of 
Theorem~\ref{multi-infinite-stab} (b) for the open right half-plane 
$\bH_{\frac{\pi}{2}}$ is that the {\em transcendental symbol of} $T$ {\em 
with respect to} $\bH_{\frac{\pi}{2}}$, i.e., the formal power series 
$$
T[e^{z\cdot w}]:=\sum_{\alpha \in \NN^n}T(z^\alpha)\frac {w^\alpha}{\alpha!}, 
$$
defines an entire function which is the limit, uniformly on compact sets, of 
weakly Hurwitz stable polynomials. 
\end{remark}

The analog of Theorem \ref{multi-infinite-stab} for real stability presevers
-- which solves 
Problem \ref{prob2} for $\KK=\RR$ and $\Omega=H$ -- is as follows.

\begin{theorem}\label{multi-infinite-hyp}
Let $T : \RR[z_1,\ldots, z_n] \rightarrow \RR[z_1,\ldots, z_n]$ be a linear 
operator. Then $T$ preserves real stability if and only if either
\begin{itemize}
\item[(a)] $T$ has range of dimension no greater than two and is of the form 
$$
T(f) = \alpha(f)P + \beta(f)Q,
$$
where $\alpha, \beta: \RR[z_1,\ldots, z_n] \rightarrow \RR$ are linear 
functionals and $P,Q$ are real stable polynomials such that $P \ll Q$, or 
\item[(b)] $\overline{G}_T(z,w)\in \HHH_{2n}(\RR)$, or 
\item[(c)] $\overline{G}_T(z,-w)\in \HHH_{2n}(\RR)$. 
\end{itemize}
\end{theorem}

\subsection{Fundamental Properties of Stable Polynomials}\label{s-2}

In this section we first review some basic facts and then prove several 
results on (real) stable polynomials needed later on.
We start with an immediate consequence of the definitions 
(cf.~\cite[Lemma 2.1]{BBS3}).

\begin{lemma}\label{l-obv}
Let $f\in\KK[z_1,\ldots,z_n]$, where $\KK=\RR$ or $\CC$. Then $f\in\HH_n(\KK)$
if and only if $f(\lambda t+\alpha)\in\HH_1(\KK)$ for all $\lambda\in\RR_+^n$
and $\alpha\in\RR^n$.
\end{lemma}

The following theorem is a multivariate version of Hurwitz' theorem on the 
``continuity of zeros'', see, e.g., \cite[Footnote 3, p.~96]{COSW} for a
proof.

\begin{theorem}[Hurwitz' theorem]\label{mult-hur}
Let $D$ be a domain (open connected set) in $\CC^n$ and suppose 
$\{f_k\}_{k=1}^\infty$ is a sequence of non-vanishing 
analytic functions on $D$ that converge to $f$ uniformly on compact 
subsets of $D$. Then $f$ is either 
non-vanishing on $D$ or else identically zero. 
\end{theorem}

We next list some of the closure properties for (real) stable polynomials. 
The first property in Lemma \ref{l-close} below is deduced by applying 
Theorem~\ref{mult-hur} with
$D= \{z \in \CC : \Im(z)>0\}^n$ and $f_k(z_1,\ldots,z_n)
= f(z_1, \ldots, z_{i-1}, \mu+z_i/k, z_{i+1}, \ldots, z_n)$, $k\in\NN$, while
the remaining properties are easy consequences of the definitions and 
Lemma~\ref{l-obv}.

\begin{lemma}\label{l-close}
Let $\KK=\RR$ or $\CC$ and $f\in\HH_n(\KK)$ be of degree  $d_j$ in $z_j$, 
$1\le j\le n$. Then for any $1\le i\le n$ one has:
\begin{enumerate} 
\item $f(z_1,\ldots,z_{i-1},\mu,z_{i+1},\ldots,z_n) 
\in \HH_{n-1}(\KK)\cup\{0\}$ for $\mu\in \RR$;
\item $f(z_1,\ldots,z_{i-1},\lambda z_i,z_{i+1},\ldots,z_n) 
\in \HH_{n}(\KK)$ for $\lambda>0$;
\item $z_i^{d_i}f(z_1,\ldots,z_{i-1},-z_i^{-1},z_{i+1},\ldots,z_n)\in 
\HH_{n}(\KK)$;
\item $f(z_1,\ldots,z_{i-1},z_j, z_{i+1},\ldots,z_n)\in \HH_{n-1}(\KK)$ for 
$1\le j\neq i\le n$.
\end{enumerate} 
\end{lemma}

We will now establish a series of results involving the notion of proper
position introduced in Definition~\ref{d-pp}, compare with 
\cite[\S 2]{BBS3}, \cite[Theorem~3.2]{COSW}. In particular, the next lemma 
provides a multivariate analog of the classical Hermite-Biehler 
theorem showing that proper position is a natural higher dimensional 
counterpart of interlacing. 

\begin{lemma}\label{krein}
Let $ f,g \in \RR[z_1,\ldots, z_n]\setminus \{0\}$, set $h = f+ig$ and 
suppose that $f$ and $g$ are not constant multiples of each other. 
The following  are equivalent: 
\begin{enumerate}
\item $h\in\HH_n(\CC)$, that is, $g\ll f$;
\item $|h(z)|>|h(\bar{z})|$ for all $z=(z_1,\ldots,z_n)\in\CC^n$ with 
$\Im(z_j)>0$, $1\le j\le n$, where $\bar{z}=(\bar{z}_1,\ldots,\bar{z}_n)$; 
\item $f+z_{n+1}g\in\HH_{n+1}(\RR)$;
\item $f,g\in\HH_n(\RR)$ and 
$$
\text{{\em Im}}\!\left(\frac{f(z)}{g(z)}\right)\ge 0
$$
whenever $z=(z_1,\ldots,z_n)\in\CC^n$ with $\Im(z_j)>0$, $1\le j\le n$;
\item $g(\lambda t+\alpha)\ll f(\lambda t+\alpha)$ for all $\lambda\in\RR_+^n$
and $\alpha\in\RR^n$.
\end{enumerate}
\end{lemma}

\begin{proof} 
The equivalences (1) $\Leftrightarrow$ (5) and (3) $\Leftrightarrow$ (4) are 
immediate from the definitions and 
Lemma~\ref{l-obv}. It is also clear that 
(3) $\Rightarrow$ (1) and (2) $\Rightarrow$ (1). We proceed to 
prove the remaining implications, namely 
(1) $\Rightarrow$ (3) and (1) $\Rightarrow$ (2). 

Suppose that (1) holds, i.e., $h\in\HH_n(\CC)$, and fix $w=\alpha + i\beta \in 
\{\zeta \in \CC: \Im(\zeta)>0\}^n$ with $\alpha,\beta\in\RR^n$. The polynomial 
$p(t)= h(\alpha+t\beta)$ is stable so we may write it as 
$$
p(t)= C\prod_{j=1}^d(t-\zeta_j), \quad  
$$ 
where $C \in \CC$ and $\Im(\zeta_j) \leq 0$, $1\le j\le d$. 
Note that for all $j$ one has $|i-\zeta_j| \ge |-i-\zeta_j|$ with equality 
only if $\zeta_j \in \RR$. Hence 
$$
|f(w)+ig(w)| = |h(w)| \geq |h(\bar{w})|= \left|\overline{h(\bar{w})}\right|
=|f(w)-ig(w)|.  
$$
To prove (3) we need to show that the univariate polynomial 
$q(z_{n+1}) = f(w)+z_{n+1}g(w)$ is stable. If $g(w)=0$ 
then $q(z_{n+1})$ is a non-zero constant so it is stable. Therefore we 
may assume that $g(w) \neq 0$. By the above we have 
$
|f(w)/g(w)+i| \geq |f(w)/g(w)-i|, 
$
which implies that $\Im(f(w)/g(w)) \geq 0$. Now if $q(z_{n+1})=0$ then 
$f(w)/g(w) + z_{n+1} =0$
and thus $\Im(z_{n+1})\leq 0$. It follows that $q\in\HH_1(\CC)$ and since
$w\in\{\zeta \in \CC: \Im(\zeta)>0\}^n$ was arbitrarily fixed we deduce that
$f+z_{n+1}g\in\HH_{n+1}(\CC)$ hence $f+z_{n+1}g\in\HH_{n+1}(\RR)$, 
which confirms (3). 

Finally, to show that (1) $\Rightarrow$ (2) note that by the above it is actually 
enough to prove that (3) $\Rightarrow$ (2). Now if (3) holds then letting $z_{n+1}=i$ we get in particular that 
$h\in\HH_n(\CC)$ and therefore $|f(w)+ig(w)|\ge |f(w)-ig(w)|$ for all 
$w\in\{\zeta \in \CC: \Im(\zeta)>0\}^n$. It remains to show that we cannot 
have $|f(w)+ig(w)| =|f(w)-ig(w)|$ for such $w$. Supposing the contrary it 
follows from the above arguments that for some $\alpha\in\RR^n$ and 
$\beta\in\RR_+^n$ all the zeros $\zeta_j$, $1\le j\le d$, of
the polynomial $p(t)=f(\alpha+t\beta) + i g(\alpha+t\beta)$ are real. Since
$f(\alpha+t\beta)$ and $g(\alpha+t\beta)$ have real coefficients this can 
occur only if these polynomials are constant multiples of each other, say 
$f(\alpha+t\beta) + a g(\alpha+t\beta)=0$ for some $a\in\RR$. By setting the 
variable $z_{n+1}=a$ we see that $f(z) + a g(z)$ is either a stable 
polynomial or identically zero. Since it vanishes for $z=\alpha+i\beta$ it 
must be identically zero, which contradicts the assumption that $f$ and $g$ 
are not constant multiples of each other. We conclude that 
(3) $\Rightarrow$ (2), which 
completes the proof of the lemma.
\end{proof}

\begin{remark}\label{r-krein}
The equivalence (3) $\Leftrightarrow$ (4) in Lemma \ref{krein} extends in
obvious fashion to $\bH_\theta$-stable polynomials with complex coefficients: 
if 
$f,g\in\CC[z_1,\ldots,z_n]\setminus \{0\}$ then $f+z_{n+1}g$ is 
$\bH_\theta$-stable (in $n+1$ variables) if and only if
$$
\Im\!\left(\frac{e^{i\theta}f(z)}{g(z)}\right)\ge 0,
\quad z=(z_1,\ldots,z_n)\in\bH_\theta^n.
$$
Note in particular that for the open right-half plane 
$\bH_{\frac{\pi}{2}}$ the above relation reads
$$
\Re\!\left(\frac{f(z)}{g(z)}\right)\ge 0,
\quad \Re(z_j)>0,\,1\le j\le n.
$$
\end{remark}

\begin{definition}\label{j-W}
For $g,f \in \RR[z_1,\ldots, z_n]$ and $1 \leq j \leq n$ define the $j$-th 
{\em Wronskian} of $g,f$  as 
$$
W_j[g,f] 
=\frac{\partial g}{\partial z_j}\cdot f-g\cdot\frac{\partial f}{\partial z_j}.
$$
\end{definition}

The following multivariate analog of the Hermite-Kakeya-Obreschkoff 
theorem (see, e.g., \cite[Theorem 6.3.8]{RS} for the classical univariate 
version) was first proved in 
\cite{BBS3}. Below we give an alternative shorter proof. 

\begin{theorem}\label{th-HKO}
Let $f,g \in \RR[z_1,\ldots, z_n]$. All non-zero polynomials in the space 
$$
\{\alpha f + \beta g : \alpha, \beta \in \RR\}
$$
are stable if and only if either $f=g\equiv 0$, $f \ll g$ or $g\ll f$. 
Moreover, if $g \ll f$ then  $W_j[g,f](x) \leq 0$  for all $x \in \RR^n$ and 
$1\leq j \leq n$, and if 
$f \ll g$ then  $W_j[g,f](x) \geq 0$  for all $x \in \RR^n$ and 
$1\leq j \leq n$. 
\end{theorem}

\begin{proof}
If $f$ and $g$ are linearly dependent then both conditions say that $f$ and $g$ are real stable or zero and then the Wronskians are all zero. Hence we may assume that $f,g \in \RR[z_1,\ldots, z_n]\setminus \{0\}$ are not constant multiples of each other. 

The ``if'' direction follows from Lemma~\ref{krein} (1) $\Rightarrow$ (3) and Lemma \ref{l-close} (1). 

For the converse assume that all non-zero polynomials in the space 
$
\{\alpha f + \beta g : \alpha, \beta \in \RR\}
$
are stable  and that there exist 
$z,w\in \{\zeta \in \CC: \Im(\zeta)>0\}^n$ for which 
$$
\Im\!\left(\frac{f(z)}{g(z)}\right) > 0 \quad \mbox{ and } \quad  
\Im\!\left(\frac{f(w)}{g(w)}\right) < 0. 
$$
By connectivity there is a number 
$v \in \{\zeta \in \CC: \Im(\zeta)>0\}^n$ for which 
$f(v)/g(v) = a \in \RR \setminus \{0\}$. It follows that $f-ag$ is identically zero which we assumed was not the case. Hence 
$\Im(f(z)/g(z))$ has constant sign for 
$z  \in \{\zeta \in \CC: \Im(\zeta)>0\}^n$ and  we conclude, by 
Lemma~\ref{krein} (4) $\Rightarrow$ (1),  that 
either $f \ll g$ or $g \ll f$. 

Suppose now that $g \ll f$, where we may assume that $g$ is not identically 
zero. Let $x \in \RR^n$ be such that $g(x) \neq 0$ and consider the rational 
function $q(t)= f(x+e_jt)/g(x+e_jt)$, where $e_j$ is the $j$-th standard 
basis vector of $\RR^n$. Hence $q(t)$ is analytic at
 the origin with a first order Taylor expansion given by 
$$
q(t)=\frac{f(x)}{g(x)}- \frac {W_j[g,f](x)}{g^2(x)}t + O(t^2). 
$$
From Lemma \ref{krein} and the fact that $\Im(f(x)/g(x))=0$ 
we then get $W_j[g,f](x) \leq 0$. Since the set of 
all $x \in \RR^n$ for which 
$g(x) \neq 0$ is dense in $\RR^n$ the theorem follows. 
\end{proof}

\begin{corollary}\label{cor-HB}
If $f,g \in \RR[z_1,\ldots, z_n]\setminus \{0\}$ are real stable polynomials
such that $f\ll g$ and $g\ll f$ then $f=\alpha g$ for some $\alpha\in\RR$.
\end{corollary}

\begin{proof}
If $f,g \in \RR[z_1,\ldots, z_n]\setminus \{0\}$ are real stable polynomials
such that $f\ll g$ and $g\ll f$ then $f \ll g$ and $f \ll -g$. Hence, by Theorem \ref{th-HKO}, 
$$
\frac \partial {\partial z_j} \left( \frac g f \right)(x) = \frac {W_j[g,f](x)}{f(x)^2}=0
$$
for all $1\leq j \leq n$ and $x \in \RR^n$ for which $f(x) \neq 0$. Thus $g/f$ is a real constant. 
\end{proof}

\section{Hard-Core/Algebraic Classification: Sufficiency}\label{s-6}

\subsection{Multi-affine Polynomials and the Lieb-Sokal Lemma}\label{s-5}
A polynomial $f \in \CC[z_1,\ldots, z_n]$ is {\em multi-affine} if all 
monomials in its Taylor expansion are square-free, i.e., if 
$f$ can be written as 
$$
f(z)= \sum_{S \subseteq [n]} a(S) z^S, \text{ where } 
z^S := \prod_{i \in S} z_i,\ \  a(S) \in \CC, \ \ [n]=\{1,\ldots,n\}. 
$$
Hence $f \in \CC[z_1,\ldots, z_n]$ is multi-affine if and only if 
$f \in \CC_{(1^n)}[z_1,\ldots, z_n]$, 
where $(1^n)=(1,\ldots, 1) \in \NN^n$. 

The following lemma (in the case $n=1$) is due to Lieb and Sokal 
\cite[Lemma~2.3]{LS}. For completeness we provide here a short proof 
essentially based on the same idea.
  
\begin{lemma}[Lieb-Sokal]\label{Li-So}
Let $P(z)+wQ(z) \in \CC[z_1,\ldots,z_n,w]$ be stable. 
If the degree in the variable $z_j$ is at most one then the polynomial 
$$
P(z)-\frac {\partial Q(z)}{\partial z_j}
$$
is either identically zero or stable. 
\end{lemma}

\begin{proof}
We may assume that $Q(z)$ is not identically equal to zero and 
that $j=1$. 
Since $Q(z)$ is stable and $\Im(w)>0\Leftrightarrow \Im(-w^{-1})>0$ 
the polynomial  
$$
wQ(z)-\frac{\partial Q(z)}{\partial z_1}
= wQ(z_1-w^{-1},z_2,\ldots, z_n)
$$ 
is stable. Hence, by Remark \ref{r-krein} (with $\theta=0$) one has
$$
\Im\!\left(\frac {P(z)-\partial Q(z)/\partial z_1}{Q(z)}\right) 
= \Im\!\left(\frac {P(z)}{Q(z)}\right) 
+ \Im\left(\frac {-\partial Q(z)/\partial z_1}{Q(z)}\right)  \geq 0
$$
for all $z \in \{\zeta \in \CC : \Im(\zeta)>0\}^n$, so by 
Remark \ref{r-krein} again the  
polynomial 
$$
P(z)-\frac{\partial Q(z)}{\partial z_1}+wQ(z)
$$
is stable. In light of Lemma~\ref{l-close} (1), this implies in particular 
the conclusion of the lemma (letting $w=0$).
\end{proof}

We can now settle the sufficiency part of Theorem~\ref{multi-finite-stab} 
in the
case of multi-affine polynomials, i.e., for $\kappa=(1,\ldots,1)$. 

\begin{lemma}\label{ma-case}
Let $T : \CC_{(1^n)}[z_1,\ldots,z_n] \rightarrow 
\CC[z_1,\ldots,z_n]$ be a
linear operator such that
$$
G_{T}(z,w):=T\!\left[(z+w)^{[n]}\right] 
=\sum_{S \subseteq [n]} T\left[z^S\right]w^{[n] 
\setminus S}  
$$
is stable, where $z=(z_1,\ldots,z_m)$ and $w=(w_1,\ldots,w_n)$. 
Then $T$ preserves stability. 
\end{lemma}

\begin{proof}
Let $T$ be as in the lemma. Since 
$\Im(w_j)>0 \Leftrightarrow \Im(-w_j^{-1}) >0$ we have  that 
$$w^{[n]}G_T(z,-w^{-1}) \in \HH_{2n}(\CC)\Longleftrightarrow 
G_T(z,w) \in  \HH_{2n}(\CC).$$
Therefore, if $f\in\CC[v_1,\ldots,v_n]$ is stable and multi-affine we deduce 
that
$$
w^{[n]}G_T(z,-w^{-1})f(v)
= \sum_{S \subseteq [n]} T\left[z^S\right](-w)^{S}f(v) 
\in \HH_{3n}(\CC),
$$
where $v=(v_1,\ldots,v_n)$.
By repeated use of Lieb-Sokal's Lemma \ref{Li-So} we then get
$$
\sum_{S \subseteq [n]} T\left[z^S\right]f^{(S)}(v) 
\in \HH_{2n}(\CC)\cup\{0\}.
$$
Letting $\{\zeta \in \CC: \Im(\zeta)>0\}^n\ni v\to 0$ and invoking Hurwitz' 
Theorem \ref{mult-hur} we obtain
$$
T(f)=\sum_{S \subseteq [n]} T\left[z^S\right]f^{(S)}(0) 
\in \HH_{n}(\CC)\cup\{0\},
$$
which proves the lemma.
\end{proof}

\subsection{Decoupling Schemes: Polarizations of Polynomials and 
Operators}\label{ss-61}

For $\kappa =(\kappa_1,\ldots,\kappa_n)\in \NN^n$ let $\CM^\kappa$ be the 
space of multi-affine 
polynomials in the variables 
$\{ z_{ij} : 1\leq i \leq n, 1\leq j \leq \kappa_i\}$. 
Define a (linear) {\em polarization operator} 
$$
\Pi_\kappa^\uparrow : \CC_\kappa[z_1, \ldots, z_n] \rightarrow \CM^\kappa
$$ 
that associates to each $f\in \CC_\kappa[z_1, \ldots, z_n]$ 
the unique polynomial 
$\Pi_\kappa^\uparrow(f) \in \CM^\kappa$ such that 
\begin{itemize}
\item[(a)] for any $1 \leq i \leq n$ the polynomial  
$\Pi_\kappa^\uparrow(f)$ is symmetric in 
$\{ z_{ij} : 1\leq j \leq \kappa_i\}$; 
\item[(b)] if we let  $z_{ij}=z_i$ for all $1\le i\le n$ and 
$1\le j\le \kappa_i$
in $\Pi_\kappa^\uparrow(f)$ we recover $f$. 
\end{itemize}
In other words, if $\alpha \leq \kappa$ then  
$$
\Pi_\kappa^\uparrow(z^\alpha)= {\binom {\kappa} {\alpha} }^{-1} 
E_{\alpha_1}(z_{11},\ldots, z_{1\kappa_1})\cdots 
E_{\alpha_n}(z_{n1},\ldots, z_{n\kappa_n}),
$$
where $E_i(x_1,\ldots, x_m)$ is the $i$-th elementary symmetric polynomial in 
the variables $x_1, \ldots, x_m$, that is,
$$
E_i(x_1,\ldots, x_m)= \sum_{S\subseteq [m],\, |S|=i}x^S, \quad 0\leq i \leq m.
$$

Dually we define a (linear) {\em projection operator} 
$$
\Pi_\kappa^\downarrow : \CM^\kappa \rightarrow  
\CC_\kappa[z_1, \ldots, z_n]
$$  
by letting $z_{ij} \mapsto z_i$ and extending linearly. Note that by (b) 
above $\Pi_\kappa^\downarrow \circ \Pi_\kappa^\uparrow$ is the identity 
operator 
on $\CC_\kappa[z_1, \ldots, z_n]$ while 
$\Pi_\kappa^\uparrow \circ \Pi_\kappa^\downarrow$ is the operator on 
$\CM^\kappa$ that 
for each $1 \leq i \leq n$ {\em symmetrizes} all the variables in 
$\{ z_{ij} : 1\leq j \leq \kappa_i\}$.

\begin{remark}\label{decoupling}
In physics literature polarization and projection operators as above 
have mostly been used for univariate polynomials and 
are often referred to as ``decoupling procedures'' \cite{beau,LY}. Such 
procedures combined with P\'olya type results \cite[Hilfssatz II]{pol-riem}
were employed by Lee and Yang in their original proof of 
the ``circle theorem'' \cite[Appendix II]{LY} (see also Kac's comments in
\cite[pp.~424--426]{pol-coll} and \cite[\S 9]{BB-II}). Multivariate 
polarization notions similar to 
those defined above are used in the 
theory of hyperbolic polynomials and partial differential equations 
\cite{garding,hor2,Lax}.
\end{remark}

Let  $T: \CC_\kappa[z_1,\ldots, z_n] \rightarrow  
\CC_\gamma [z_1,\ldots, z_n]$ be a linear 
operator. The {\em polarization} of $T$ is defined as the linear operator
$\Pi(T) : \CM^\kappa \rightarrow \CM^\gamma$ given by 
\begin{equation}\label{polop}
\Pi(T)= \Pi_\gamma^\uparrow \circ T\circ \Pi_\kappa^\downarrow, 
\end{equation}
and conversely we have 
\begin{equation}\label{poldown}
T= \Pi_\gamma^\downarrow \circ \Pi(T)\circ \Pi_\kappa^\uparrow. 
\end{equation}

It is immediate from Lemma~\ref{l-close} (4) that the projection operator 
$\Pi_\kappa^\downarrow$ preserves stability, but the remarkable fact is that 
so does  $\Pi_\kappa^\uparrow$. This is  essentially the famous 
Grace-Walsh-Szeg\"o coincidence theorem \cite{grace,szego,walsh} albeit in 
a disguised form. For completeness let us state this theorem. 

\begin{theorem}[Grace-Walsh-Szeg\"o]\label{GWS}
Let $f$ be a symmetric multi-affine polynomial in $n$ complex variables, 
let $C$ be an open or closed circular domain, and let $\xi_1, \ldots, \xi_n$ 
be points in $C$. Suppose further that either the total degree of $f$ equals 
$n$ or that $C$ is convex (or both). Then there exists at least one point 
$\xi \in C$ such that 
\begin{equation}\label{ssss}
f(\xi_1, \ldots, \xi_n)= f(\xi, \ldots, \xi). 
\end{equation}
\end{theorem}

We give a new and self-contained 
proof of Theorem \ref{GWS} in \cite[\S 2]{BB-II}.

\begin{proposition}\label{gws}
Let $f \in \CC_\kappa[z_1,\ldots, z_n]$. Then $f$ is stable if and only 
if $\Pi_\kappa^\uparrow(f)$ is stable. 
\end{proposition}

\begin{proof}
The ``if'' direction follows from the fact that  $\Pi_\kappa^\downarrow$ 
preserves stability. For the other direction it suffices to prove the theorem 
for univariate polynomials since we may polarize the variables one at a time 
(when the other variables are fixed in the upper half-plane). Suppose therefore
that $f(z)=\sum_{k=0}^d a_k z^k$ is a univariate stable polynomial. Then the 
polarization of $f$ is given by 
$$
F(z_1,\ldots, z_n)= \sum_{k=0}^d a_k {\binom d k}^{-1}E_k(z_1,\ldots, z_d). 
$$
If $F(\zeta_1,\ldots, \zeta_n) = 0$ for some $(\zeta_1,\ldots, \zeta_n) 
\in H^n$ then by the Grace-Walsh-Szeg\"o theorem (Theorem \ref{GWS}) there 
exists $\zeta \in H$ such that $f(\zeta) = F(\zeta,\ldots, \zeta) = 0$, 
contradicting the stability of $f$. 
\end{proof}
 
\begin{lemma}\label{polarized-symbol}
Let $T: \CC_\kappa[z_1,\ldots, z_n] \rightarrow  
\CC_\gamma [z_1,\ldots, z_n]$ be a linear operator. Then the symbol of the 
polarization of $T$ is the polarization of the symbol of $T$, that is,
$$
G_{\Pi(T)} = \Pi_{\gamma \oplus \kappa}^\uparrow(G_T), 
$$
where $\Pi_{\gamma \oplus \kappa}^\uparrow : 
\CC_{\gamma \oplus \kappa}[z_1,\ldots,z_n,w_1,\ldots, w_n] \rightarrow 
\CM^{\gamma \oplus \kappa}$. 
\end{lemma}

\begin{proof} 
We have that 
$$
\Pi_\kappa^{\downarrow z}\!\left[\prod_{i=1}^n\prod_{j=1}^{\kappa_i}
(z_{ij}+w_{ij})\right]=  
\Pi_\kappa^{\downarrow z}\Big[ (\Pi_{\kappa}^{\uparrow z}\circ 
\Pi_{\kappa}^{\uparrow w})\big[(z+w)^\kappa\big]\Big], 
$$
where the upper indices $z$ and $w$ indicate that the corresponding 
operators act only on the $z$-variables and 
$w$-variables, respectively. Hence the symbol of the polarized operator 
may be written as
\begin{eqnarray*}
G_{\Pi(T)}&=& (\Pi_\gamma^{\uparrow z}\circ T\circ \Pi_\kappa^{\downarrow z})
\left[ \prod_{i=1}^n\prod_{j=1}^{\kappa_i}(z_{ij}+w_{ij})\right] \\
&=&(\Pi_\gamma^{\uparrow z}\circ T\circ \Pi_\kappa^{\downarrow z})
\Big[ (\Pi_{\kappa}^{\uparrow z}\circ 
\Pi_{\kappa}^{\uparrow w})\big[(z+w)^\kappa\big]\Big].
\end{eqnarray*}
Using the fact that $\Pi_\kappa^{\downarrow z}\circ\Pi_{\kappa}^{\uparrow z}$  
is the identity operator and that $T$ and $\Pi_{\kappa}^{\uparrow w}$ commute 
(since $T$ only acts on the $z$-variables) we get 
$$
G_{\Pi(T)}= (\Pi_\gamma^{\uparrow z} \circ \Pi_{\kappa}^{\uparrow w}\circ T)
\big[(z+w)^\kappa\big]=\Pi_{\gamma \oplus \kappa}^\uparrow(G_T),
$$
as desired.
\end{proof}

\begin{proof}[Proof of Sufficiency in Theorem~\ref{multi-finite-stab}] 
If $G_T$ is stable then so is $G_{\Pi(T)}$ by Lemma~\ref{polarized-symbol} 
and Proposition~\ref{gws}. 
From Lemma~\ref{ma-case} we deduce that $\Pi(T)$ preserves stability, and then
by equation \eqref{poldown} so does $T$ since in view of 
Lemma~\ref{l-close} and 
Proposition~\ref{gws} both operators $\Pi_\gamma^\downarrow$ and 
$\Pi_\kappa^\uparrow$ preserve stability.
\end{proof}

\section{Algebraic Classification: Necessity in the Complex Case}\label{s-4}

We will need the following simple lemma which makes it clear that there are 
plenty of stable polynomials. 

\begin{lemma}\label{plenty}  
Let $f \in \CC_\kappa[z_1,\ldots, z_n]$, where 
$\kappa=(\kappa_1,\ldots,\kappa_n)\in\NN^n$, and 
$W=(W_1,\ldots,W_n) \in \{\zeta \in \CC: \Im(\zeta)>0\}^n$. Then  
for all $\epsilon>0$ sufficiently small one has
$$
(z+W)^\kappa + \epsilon f(z) \in \HH_n(\CC), \,\text{ where }\,
(z+W)^\kappa=(z_1+W_1)^{\kappa_1}\cdots (z_n+W_n)^{\kappa_n}.
$$
\end{lemma}

\begin{proof}
Set $Y=(\Im(W_1), \ldots, \Im(W_n))\in\RR_+^n$. For  
$\alpha \leq \kappa$ we then have 
$$
\left|\frac{(z+W)^\alpha}{(z+W)^\kappa}\right| \leq Y^{\alpha-\kappa}\,\text{ if }\, 
z \in \{ \zeta \in \CC : \Im(\zeta) \geq 0 \}^n. 
$$ 
Expanding $f$ in powers of 
$z+W$ we see that  there exists $\epsilon_0 >0$ such that 
$$
\frac {|f(z)|}{|(z+W)^\kappa|}< \frac{1}{\epsilon_0}\,\text{ if }\,
z \in \{ \zeta \in \CC : \Im(\zeta) \geq 0 \}^n. 
$$
Hence 
$
|(z+W)^\kappa + \epsilon f(z)| \geq |(z+W)^\kappa| - \epsilon |f(z)| >0$ for 
all  $z \in \{ \zeta \in \CC : \Im(\zeta) \geq 0 \}^n$ and 
$\epsilon\in(0,\epsilon_0)$, so that 
in particular $(z+W)^\kappa + \epsilon f(z) \in  \HH_n(\CC)$ for all such
$\epsilon$. 
\end{proof}

\begin{remark}
From the proof of Lemma \ref{plenty} it follows that the topological 
dimension of the set $\HH_n(\CC)\cap \CC_\kappa[z_1,\ldots, z_n]$ (of all 
stable polynomials in $n$ variables of degree at most $\kappa$) equals 
$(1+\kappa_1)\cdots (1+\kappa_n)$. In particular, this explains the 
statement made in \cite[\S 1]{BBL} concerning the topological dimension of 
the set of so-called strongly Rayleigh probability measures on the Boolean 
algebra $2^{[n]}$, where $[n]=\{1,\ldots,n\}$.
\end{remark}

\begin{lemma}\label{spaces}
Let $V \subseteq \KK[z_1,\ldots, z_n]$ be a $\KK$-linear space, 
where $\KK=\RR$ or $\CC$.
\begin{itemize}
\item[(i)] If $\KK=\RR$ and every non-zero element of $V$ is real stable 
then $\dim V \leq 2$.
\item[(ii)]
If $\KK=\CC$ and every non-zero element of $V$ is stable 
then $\dim V \leq 1$.
\end{itemize}
\end{lemma}

\begin{proof}
We first deal with the real case. Suppose that there are three linearly 
independent polynomials $f_1,f_2$  and $f_3$ 
in $V$. By the multivariate Hermite-Kakeya-Obreschkoff theorem 
(Theorem \ref{th-HKO}) and the assumption on $V$ 
these polynomials are mutually in proper position. 
Without loss of generality we may assume that 
$f_1 \ll f_2$ and $f_1 \gg f_3$. Now consider the line  
segment in $V$ given by $\ell_\theta =\theta f_3 + (1-\theta) f_2$, 
$0 \leq \theta \leq 1$, and note in particular that for any such $\theta$ one 
has either $f_1\ll \ell_\theta$ or $f_1\gg \ell_\theta$.
Set $\eta=\sup\{\theta\in[0,1]:f_1\ll \ell_\theta\}$.
Since $f_1 \ll \ell_0$ and $f_1 \gg \ell_1$ it follows from Hurwitz' theorem  
(Theorem \ref{mult-hur}) that $f_1 \ll \ell_\eta$ and $f_1 \gg \ell_\eta$. 
This means that $f_1$ and $\ell_\eta$ are constant multiples of each other
(cf.~Corollary \ref{cor-HB}), 
contrary to the assumption that $f_1,f_2$  and $f_3$ are linearly independent. 

For the complex case 
let 
$$
V_R= \{ p : p+iq \in V\text{ with } p,q\in \RR[z_1,\ldots,z_n] \}
$$ 
be the ``real component'' of $V$. By Theorem~\ref{th-HKO} all 
non-zero polynomials 
in $V_R$ are real stable
hence $\dim_{\RR} V_R \leq 2$ by part (i) proved above. Note that 
$V$ is the complex span of $V_R$. If $\dim_{\RR} V_R \leq 1$ we are done so 
we may assume that 
$\{p,q\}$ is a basis of $V_R$ with $f:=p+iq \in V$ which by assumption is
a (not identically zero) stable polynomial. By 
Theorem~\ref{th-HKO} again we have $W_j[p,q](x) \geq 0$, $1\leq j \leq n$, 
$x\in\RR^n$, and since $p$ and $q$ are linearly 
independent the $j_0$-th Wronskian $W_{j_0}[p,q]$ 
is not identically zero for some index $j_0$. 
Now if $g\in V\setminus\{0,f\}$ we may write 
$$
g= ap+bq+i(cp+dq) 
$$
for some $a,b,c,d \in \RR$. We have to show that $g$ is a (complex) constant 
multiple of $f$. Since $g \in V$ is not identically zero it is a stable 
polynomial, so that in particular 
$$
W_{j_0}[ap+bq,cp+dq] = (ad-bc)W_{j_0}[p,q] \geq 0
$$
(cf.~Theorem~\ref{th-HKO}). Since $W_{j_0}[p,q] \geq 0$ and 
$W_{j_0}[p,q]\not\equiv 0$ 
it follows that $ad-bc \geq 0$. 

Let $u,v\in\RR$ and note that by linearity we have 
$$g+(u+iv)f =(a+u)p+(b-v)q+i( (c+v)p+(d+u)q)    \in V.
$$ 
Arguing as above we deduce that 
$$
H(u,v):=(a+u)(d+u)-(b-v)(c+v) \geq 0 
$$
for all $u,v \in \RR$. But 
$$
4H(u,v)=(2u+a+d)^2+(2v+c-b)^2-(a-d)^2-(b+c)^2,
$$
so $H(u,v) \geq 0$ for all $u,v \in \RR$ if and only if $a=d$ and $b=-c$. 
This gives  
$$
g=ap-cq+i(cp+aq)=(a+ic)(p+iq)=(a+ic)f,
$$
as was to be shown.
\end{proof}

We can now show that the symbol of any stability preserver whose 
range is not one-dimensional must necessarily be stable. 

\begin{proof}[Proof of Necessity in Theorem~\ref{multi-finite-stab}]
Suppose that 
$T : \CC_\kappa[z_1,\ldots, z_n] \rightarrow \CC[z_1,\ldots,z_n]$  
preserves stability. 
Given $W \in \{\zeta \in \CC: \Im(\zeta)>0\}^n$  we have that 
$T[(z+W)^\kappa] \in \HH_{n}(\CC)\cup\{0\}$. Assume first that there exists   
$W \in \{\zeta \in \CC: \Im(\zeta)>0\}^n$ for which $T[(z+W)^\kappa]\equiv 0$
and let $f \in \CC_\kappa[z_1,\ldots, z_n]$. Then by Lemma \ref{plenty} 
there is an $\epsilon>0$ such that 
$(z+W)^\kappa + \epsilon f(z) \in  \HH_n(\CC)$.   
It follows that 
$$
\epsilon T[f(z)]=T[(z+W)^\kappa + \epsilon f(z)] \in \HH_n(\CC)\cup \{0\}.
$$
Hence the image of $T$ is a linear space whose non-zero elements are all 
stable polynomials. By Lemma~\ref{spaces} (ii) we have that 
$T$ is given by $T(f) =\alpha(f)P$, where $P$ is a stable polynomial and 
$\alpha$ is a linear functional. 

On the other hand, if $T[(z+W)^\kappa]\not\equiv 0$ for 
$W \in \{\zeta \in \CC: \Im(\zeta)>0\}^n$ then 
$T[(z+W)^\kappa] \in \HH_{n}(\CC)$ for all such $W$, so that 
$G_T(z,w)=T[(z+w)^\kappa] \in \HH_{2n}(\CC)$. 
\end{proof}

\section{Real Stability Preservers}\label{s-7}

The aim of this section is to prove Theorem~\ref{multi-finite-hyp}.
Let 
$$
\HH_n^{-}(\CC)= \{ f(-z): f(z) \in \HH_n(\CC)\}.
$$ 

\begin{proposition}\label{intersection}
For any $n\in\NN$ the following holds:
$$\HH_n(\CC) \cap \HH_n^{-}(\CC)= \CC \HH_n(\RR) 
:= \{ cf: c \in \CC, f \in \HH_n(\RR) \}.$$
\end{proposition}

\begin{proof}
Let $h,g \in \RR[z_1,\ldots, z_n]$ be such that $h+ig \in \HH_n(\CC)\cap 
\HH_n^{-}(\CC)$ and set $f= h-ig$. If $x,y \in \RR^n$ then 
\begin{equation}\label{extra-eq}
\overline{f(x+iy)}= h(-(-x+iy))+ig(-(-x+iy)).
\end{equation}
Since $h+ig \in \HH_n^{-}(\CC)$, the right-hand side of 
\eqref{extra-eq} is non-vanishing whenever $y\in\RR_+^n$, which 
amounts to saying that $f(x+iy)\neq 0$ for all 
$y\in\RR_+^n$. By Definition \ref{d-pp} this implies that 
$-g \ll h$, i.e., $h \ll g$. On the other hand, $g\ll h$ since 
$h+ig \in \HH_n(\CC)$. Therefore the 
proposition is just a reformulation of Corollary~\ref{cor-HB}. 
\end{proof}
 
\begin{proof}[Proof of Theorem~\ref{multi-finite-hyp}]
If $T:\RR_\kappa[z_1,\ldots, z_n] \rightarrow \RR[z_1,\ldots, z_n]$ 
is a linear operator as in part (a) of Theorem~\ref{multi-finite-hyp} then it 
preserves real stability (by Theorem \ref{th-HKO}). Moreover, 
if $T,T' : \RR_\kappa[z_1,\ldots, z_n] \rightarrow \RR[z_1,\ldots, z_n]$
are linear operators whose symbols satisfy $G_{T'}(z,w)=G_T(z,-w)$ 
then these operators are related by 
$T'(f)(z)=(-1)^{\kappa}T(f(-z))$, hence they preserve real stability simultaneously. 
To settle the sufficiency part of the theorem we may therefore 
assume that $G_T(z,w)$ is real stable. But then the desired conclusion 
simply follows from the complex case. 

To prove the necessity part of Theorem~\ref{multi-finite-hyp} consider 
$G_T(z,w)=T[(z+w)^\kappa]$ and let $W \in \{z \in \CC: \Im(z)>0\}^n$. 
Write $(z+W)^\kappa \in \HH_n(\CC)$ as 
$$
(z+W)^\kappa = F(z)+iG(z),  \quad F,G \in \RR_\kappa[z_1,\ldots, z_n].  
$$
By the multivariate 
Hermite-Kakeya-Obreschkoff theorem (Theorem \ref{th-HKO})
we have that $\alpha F + \beta G \in \HH_n(\RR)\cup\{0\}$ for all 
$\alpha, \beta \in \RR$ and hence that  
$\alpha T(F)+ \beta T(G) \in \HH_n(\RR)\cup\{0\}$ for all 
$\alpha, \beta \in \RR$. By Theorem \ref{th-HKO} again this means that 
$$
T(F)+iT(G)=G_T(z,W) \in \HH_n(\CC)\cup \HH_n^{-}(\CC) \cup \{0\}. 
$$ 
Suppose that there are $W_1, W_2  \in \{z \in \CC: \Im(z)>0\}^n$ such that 
$G_T(z,W_1) \in \HH_n(\CC) \cup \{0\}$ and 
$G_T(z,W_2) \in \HH_n^{-}(\CC) \cup \{0\}$. By a homotopy argument we then
deduce that there exists $t\in[0,1]$ such that
$$
G_T(z,W')=\HH_n(\CC)\cap \HH_n^{-}(\CC) \cup \{0\},$$
where $W'=(1-t)W_1+tW_2 \in \{z \in \CC: \Im(z)>0\}^n$. Therefore, 
by Proposition~\ref{intersection} there is a real stable polynomial $p$ and 
$a,b \in \RR$ such that 
$G_T(z,W')=(a+bi)p(z)$. Let us write $(z+W')^\kappa= f(z)+ig(z)$, 
where $f,g$ have real coefficients. 
Then $T(f)=ap$ and $T(g)=bp$ so that 
$$
T(bf-ag)=0. 
$$
Now, as noted in the proof of Lemma~\ref{plenty} for any 
$h \in \RR_\kappa[z_1,\ldots,z_n]$ there is an $\epsilon >0$ such that 
$$
\frac {|h|}{|(z+W')^\kappa|} < \frac{1}{\epsilon}
$$
whenever $\Im(z_i) \geq 0$, $1\le i\le n$. 
Hence, if $a^2+b^2 \neq 0$ (the case $a=b=0$ follows similarly) we get 
$$
\frac {|h|}{|bf-ag+i(af+bg)|} = \frac {|h|}{|(b+ia)(z+W')^\kappa|} 
< (\epsilon|(b+ia)|)^{-1}=\frac{1}{\epsilon'}.
$$
This implies that 
$$
|bf-ag+\epsilon'h+i(af+bg)| > 0
$$
whenever $\Im(z_i) \geq 0$, $1\le i\le n$. In particular, 
$bf-ag+\epsilon'h+i(af+bg) \in \HH_n(\CC)$ so by Theorem \ref{th-HKO}
we have 
$bf-ag+\epsilon'h \in \HH_n(\RR) \cup \{0\}$. It follows that 
$$
T(h)=\frac 1 {\epsilon'}T(bf-ag+\epsilon'h) \in \HH_n(\RR)\cup\{0\}.
$$
This means that all nonzero polynomials in the image of $T$ are real stable 
and by Lemma~\ref{spaces} we conclude that the image of $T$ is of 
dimension at most two. 

Thus we may assume that $T[(z+W)^\kappa] \in \HH_n(\CC)$ for all 
$W \in \{z \in \CC: \Im(z)>0\}^n$ or 
$T[(z-W)^\kappa] \in \HH_n(\CC)$ for all $W \in \{z \in \CC: \Im(z)>0\}^n$. 
But this amounts to saying that 
$T[(z+w)^\kappa] \in \HH_{2n}(\RR)$ or $T[(z-w)^\kappa] \in \HH_{2n}(\RR)$, 
as claimed. 
\end{proof}

\section{Soft-Core/Transcendental Classifications}\label{s-8}

In this section we settle Theorems \ref{multi-infinite-stab} 
and \ref{multi-infinite-hyp}, that is, the transcendental 
characterizations of stability, respectively real stability preservers. 
Let $\KK=\CC$ or $\RR$. A linear operator $T:\KK[z_1,\ldots,z_n]\to 
\KK[z_1,\ldots,z_n]$ preserves stability if and only if the same holds for 
all its restrictions $T_\kappa:\KK_\kappa[z_1,\ldots,z_n]\to 
\KK[z_1,\ldots,z_n]$, 
$\kappa\in\NN^n$. Therefore, by Theorem \ref{multi-finite-stab}, 
Theorem \ref{multi-finite-hyp} and the definitions of the symbols
$G_{T_\kappa}(z,w)$, $\overline{G}_T(z,w)$ we see that 
Theorem \ref{multi-infinite-stab}, respectively 
Theorem \ref{multi-infinite-hyp}, would follow from the case $\KK=\CC$, 
respectively $\KK=\RR$, of the next theorem. 
As before we use the standard (product) partial order $\le$ on $\NN^n$: for  
$\alpha, \beta \in \NN^n$ one has $\alpha \leq \beta$ if 
$\alpha_i \leq \beta_i$ for all $1 \leq i \leq n$. Set 
$\alpha! = \alpha_1! \cdots \alpha_n!$ as in \S \ref{new-s-prel} and let 
$$
(\beta)_\alpha = \begin{cases} 
\frac {\beta!}{(\beta-\alpha)! } \mbox{ if } \alpha \leq \beta, \\ 
0 \mbox{ otherwise.} 
\end{cases}
$$

\begin{theorem}\label{t-trans}
Let $F(z,w)= \sum_{\alpha \in \NN^n}P_\alpha(z)w^\alpha$, where
$z=(z_1,\ldots,z_n)$ and $w=(w_1,\ldots,w_n)$, be a formal power 
series in 
$w$ with coefficients in $\KK[z_1,\ldots, z_n]$. Then 
$F(z,w) \in \HHH_{2n}(\KK)$ if and only if 
$$\sum_{\alpha \leq \beta} (\beta)_\alpha P_\alpha(z)w^\alpha 
\in \HH_{2n}(\KK) \cup\{0\}$$ 
for all $\beta \in \NN^n$. 
\end{theorem}

The proof of Theorem \ref{t-trans} requires several new ingredients and
additional results that we proceed to describe. Since the arguments are the 
same for $\KK=\RR$ and $\KK=\CC$, we will only focus on the latter case.

{\allowdisplaybreaks 
\subsection{Generalized Jensen Multipliers}\label{ss-82}

For $\alpha, \beta \in \NN^n$ let $J(\alpha, \beta)= (\beta)_\alpha 
\beta^{-\alpha}$ (we use the convention that $k^{\pm k}=1$ for $k=0$). 
For $n=1$ these are the so-called Jensen multipliers that are known
to preserve (univariate) stability \cite{CC1,Le,RS}. For 
each fixed $\beta\in\NN^n$ the sequences 
$\{J(\alpha,\beta)\}_{\alpha\le \beta}$ and 
$\{(\beta)_\alpha\}_{\alpha\le \beta}$ may be called multivariate (or 
generalized) Jensen multipliers in view of the following lemma.

\begin{lemma}\label{multi}
Let $\beta \in \NN^n$. The linear operators on $\CC[z_1,\ldots,z_n]$
defined by 
\begin{eqnarray*}
z^\alpha &\mapsto& J(\alpha, \beta)z^\alpha,\quad \alpha\in\NN^n, \\
z^\alpha &\mapsto& (\beta)_\alpha z^\alpha, \quad \alpha\in\NN^n,
\end{eqnarray*}
preserve stability. 
\end{lemma}

\begin{proof}
Fix $\beta \in \NN^n$. Since the first operator is a composition of the second
operator and a rescaling of the variables, it is enough to prove the lemma 
only for the 
second operator -- call it $T$ and denote by $T_\kappa$ its
restriction to $\CC_\kappa[z_1,\ldots,z_n]$, where $\kappa \in \NN^n$. 
By Theorem~\ref{multi-finite-stab} we need to show that 
$$
G_{T_\kappa}(z,w)=T_\kappa[(z+w)^\kappa] 
= \sum_{\alpha \leq \kappa} \binom \kappa \alpha (\beta)_\alpha 
z^\alpha w^{\kappa-\alpha} 
$$
is stable for all $\kappa \in \NN^n$. However, 
$$
G_{T_\kappa}(z,w)= \prod_{i=1}^n\left[\sum_{j=0}^{\kappa_i} j! 
\binom {\kappa_i} {j} \binom {\beta_i} {j} z_i^j w_i^{\kappa_i-j}\right],
$$
so the lemma amounts to showing that for any $m,n \in \NN$
the univariate polynomial 
$$g(t)=\sum_{j=0}^{n} j! \binom {n} j \binom {m} j t^j$$ 
is real-rooted (then necessarily with all negative roots). 
To prove this note that 
$$t^{-m}g(t)=S\!\left[t^{m}\right]\Big|_{t\to t^{-1}},$$
where $S$ is the operator on $\CC[t]$ given by 
$S=\left(1+\frac{d}{dt}\right)^n$. 
To prove the well known fact that $S$ preserves stability it suffices to 
consider the case $n=1$. The symbol of $S$ is $S[(t+w)^m]=(m+t+w)(t+w)^{m-1}$,
which is a stable polynomial. The desired conclusion now follows  
from Theorem~\ref{multi-finite-hyp}.
\end{proof}

\subsection{Multivariate Sz\'asz Principles}\label{ss-81}

A remarkable fact about univariate stable polynomials, first noted by
Sz\'asz \cite{szasz} and subsequently used by Edrei \cite{edrei}, is
that the growth of such a polynomial is controlled by its first few 
non-vanishing coefficients. We extend Sz\'asz' principle
to higher dimensions and establish estimates for the growth of 
multivariate stable polynomials.

\begin{lemma}[Sz\'asz]\label{l-sz}Suppose that 
$f(z)= 1+\sum_{i=1}^{k}a_iz^i=\prod_{j=1}^k (1+\xi_j z)
$
is stable. 
Then 
$$
\sum_{j=1}^k|\xi_j|^2 \leq 3|a_1|^2+2|a_2|. 
$$
\end{lemma} 

\begin{proof}
By assumption one has $\Im(\xi_j)\le 0$, $1\le j\le k$, hence
$$
\sum_{j=1}^{k}\Im(\xi_j)^2\le \left(\sum_{j=1}^{k}\Im(\xi_j)\right)^2
=\Im(a_1)^2.
$$
Since $\sum_{j=1}^{k}\xi_j^2=a_1^2-2a_2$ it follows that
\begin{equation*}
\begin{split}
\sum_{j=1}^k|\xi_j|^2&=\sum_{j=1}^k\left(\Re(\xi_j)^2-\Im(\xi_j)^2\right)
+2\sum_{j=1}^k\Im(\xi_j)^2\\
&=\Re\left(\sum_{j=1}^{k}\xi_j^2\right)+2\sum_{j=1}^k\Im(\xi_j)^2\\
&=\Re(a_1^2-2a_2)+2\sum_{j=1}^k\Im(\xi_j)^2\\
&\le \Re(a_1^2-2a_2)+2\Im(a_1)^2\le 3|a_1|^2+2|a_2|,
\end{split}
\end{equation*}
as claimed.
\end{proof}

Sz\'asz used his lemma to obtain the following bound for the polynomial $f$ 
in Lemma~\ref{l-sz}: 
$$
|f(z)| \leq \exp \Big( r|a_1|+ 3r^2 |a_1|^2+3r^2|a_2|  \Big), \quad |z|\leq r. 
$$

\begin{remark}
If the univariate polynomial $f$ in Lemma \ref{l-sz} is assumed to have all 
real and negative zeros then $\sum_{j=1}^k \xi_j^2 \leq a_1^2$.
Using the latter inequality in the following arguments one can get
stronger bounds for the growth of multivariate real stable polynomials with
all non-negative coefficients.
\end{remark}

As before, we let $\{e_i\}_{i=1}^{n}$ be the standard basis in $\RR^n$.

\begin{lemma}\label{edreibound}
Suppose that the polynomial 
$$
f(z)=1+ \sum_{|\beta|>0}a(\beta)z^\beta 
\in\CC[z_1,\ldots, z_n] 
$$
is stable and let
$$
A=\left[3\left(\sum_{i=1}^n |a(e_i)|\right)^2
+2\sum_{i,j=1}^n |a(e_i+e_j)|\right]^{1/2}.
$$
Then 
$$
|a(\beta)| \leq |\beta|^{-|\beta|/2} \frac{\beta^\beta}{\beta!} A^{|\beta|},
\quad \beta=(\beta_1,\ldots,\beta_n) \in\NN^n,
$$
where $|\beta|=\sum_{i=1}^{n}\beta_i$.
\end{lemma}

\begin{proof}
Let us rewrite the above polynomial as $f(z)=\sum_{\alpha}a(\alpha)z^\alpha$ 
and suppose that $\beta\neq 0$ is such that $a(\beta) \neq 0$. Since the 
operator $z^\alpha \mapsto J(\alpha, \beta)z^\alpha$, $\alpha\in\NN^n$, 
preserves stability (Lemma~\ref{multi}) the polynomial 
$$
J(\beta,\beta)a(\beta)z^\beta 
+ \sum_{0<|\alpha|<|\beta|} J(\alpha, \beta)a(\alpha)z^\alpha+1
$$
is stable. Setting all variables equal to $t$ we deduce that the univariate 
polynomial $g(t)$ of degree $k=|\beta|$ given by 
\begin{multline*}
g(t)=\frac{\beta!}{\beta^\beta}a(\beta)t^{k} 
+ \sum_{2<|\alpha|<k}J(\alpha,\beta)a(\alpha)t^{|\alpha|}\\
+\left[\sum_{i\neq j\,:\,\beta_i\beta_j \ge 1} a(e_i+e_j)
+ \sum_{i\,:\,\beta_i \geq 2} (1-1/\beta_i)a(2e_i)\right]t^2
+\sum_{i\,:\,\beta_i\geq 1} a(e_i)t +1 
\end{multline*}
is stable. If we rewrite $g(t)=\sum_{i=0}^{k}a_it^i=\prod_{j=1}^k (1+\xi_jt)$ 
then by Lemma \ref{l-sz} we get
$$
\left|\frac {\beta!}{\beta^\beta}a(\beta)\right|=\prod_{j=1}^{k}|\xi_j|
\le \left[\frac{\sum_{j=1}^{k} |\xi_j|^2}{k}\right]^{k/2}
\leq \left(3|a_1|^2+2|a_2|\right)^{k/2}  k^{-k/2} \leq A^k k^{-k/2},
$$
which is the desired estimate.
\end{proof}}

\begin{theorem}\label{multi-szasz}
Suppose that the polynomial 
$$
f(z)=1+ \sum_{|\beta|>0}a(\beta)z^\beta \in \CC[z_1,\ldots, z_n] 
$$
is stable and let
\begin{equation*}
\begin{split}
&B=2^{n-1} \frac {\sqrt{2e^2-e}} {e-1} = 2^{n-1} \cdot2.0210\ldots , \\
&C=6e^2\left(\sum_{i=1}^n |a(e_i)|\right)^2+4e^2\sum_{i,j=1}^n |a(e_i+e_j)|.
\end{split}
\end{equation*}
Then 
$$
\max \Big\{ |f(z)| : |z_i| \leq r, 1\leq i \leq n\Big\} \leq B
\displaystyle{e^{Cr^2}},\quad r\ge 0.
$$
\end{theorem}

\begin{proof}
We make use of the following simple inequalities that may be proved by 
induction or from Stirling estimates.  

\begin{equation}\label{stirling}
e^{-n} \leq \frac {n!}{n^n} \leq (en+1)e^{-n}, \quad n \geq 0.
\end{equation}
Let 
$
d(n,k)= \sum_{\beta \in \NN^n, |\beta|=k} \beta^\beta/\beta!
$
and note that $\left|\{\beta \in \NN^n : |\beta|=k\}\right|
=\binom {n+k-1}{k}$ for $n,k\in \NN$. 
 By \eqref{stirling} we have the following (rough) estimate:
$$
d(n,k) \leq \binom {n+k-1}k e^k \leq 2^{n-1+k}e^k,\quad n,k\in \NN.
$$ 
Let $A$ be as in Lemma~\ref{edreibound} and apply the same lemma:
\begin{equation}\label{bound-A}
\begin{split}
&\max\Big\{ |f(z)| : |z_i|\leq r, 1\leq i \leq n\Big\} \leq 
\sum_{\beta\in\NN^n} |\beta|^{-|\beta|/2} 
\frac{\beta^\beta}{\beta!} (Ar)^{|\beta|} \\
&= \sum_{k=0}^{\infty} d(n,k)k^{-k/2}(Ar)^k 
\leq 2^{n-1} \sum_{k=0}^{\infty} k^{-k/2}(2eAr)^k.
\end{split}
\end{equation}              
Set $R=2eAr$. 
By \eqref{stirling} and the Cauchy-Schwarz inequality  we get 
\begin{multline*}
\sum_{k=0}^{\infty} k^{-k/2} R^k \leq\sum_{k=0}^\infty \sqrt{(ek+1)e^{-k}}  
\cdot \sqrt{\frac{R^{2k}}{k!}} \\
\le \sqrt{\left(\sum_{k=0}^\infty (ek+1)e^{-k}\right)   \cdot   
\left(\sum_{k=0}^\infty \frac{R^{2k}}{k!}\right)   } 
= \frac {\sqrt{2e^2-e}} {e-1} \cdot e^{R^2/2},
\end{multline*}
which combined with \eqref{bound-A} proves the theorem.
\end{proof}
 
Given $f(z)=\sum_{\alpha}a(\alpha)z^{\alpha}\in 
\CC[z_1,\ldots, z_n]$ let 
$\supp(f)=\{\alpha\in \NN^n : a(\alpha)\neq 0\}$ be its support.  Denote by
$\MM(f)$ the set of minimal elements of $\supp(f)$ with respect to the partial 
order $\le$ on $\NN^n$. Moreover, for a fixed set $\MM \subset \NN^n$ we let 
$$
\MM_\ell= \{ \alpha +\beta  : \alpha \in \MM, \beta \in \NN^n, 
|\beta| \leq \ell\}. 
$$ 
We can now establish a multivariate Sz\'asz principle.

\begin{theorem}\label{for_trans}
Let $\MM \subset \NN^n$ be a finite non-empty set and 
$f(z)=\sum_{\alpha}a(\alpha)z^\alpha\in 
\CC[z_1,\ldots, z_n]$ be a stable polynomial with $\MM(f)=\MM$. 
Then there are constants $B$ and $C$ depending  only on   
the coefficients 
$a(\alpha)$ with $\alpha \in \MM_2$  
such that 
$$
\max \Big\{ |f(z)| : |z_i| \leq r, 1\leq i \leq n\Big\} \leq B e^{Cr^2}
$$ 
for all $r\ge 0$. Moreover, $B$ and $C$ can be chosen so that they depend  
continuously on the aforementioned set of coefficients.  
\end{theorem}

\begin{proof}
We construct $B$ and $C$ inductively over $m=n+k$, where 
$$
k=\max\{|\alpha| : \alpha \in \MM(f)\}.
$$  
If $f(0) \neq 0$ then $k=0$ and the bound follows from 
Theorem~\ref{multi-szasz}, so we may assume that $k>0$. 
Also, if $n=1$ we may just factor out a multiple of $z$ and use the fact 
that $r \leq e^{r^2}$ together with Theorem~\ref{multi-szasz}.  
Hence we may assume that $n>1$. 

Let $\alpha=(\alpha_1,\ldots,\alpha_n)\in
\MM(f)$ be such that $|\alpha|=k$ and suppose (without loss of generality)
that $\alpha_1>0$. Now 
$$
f(z_1,\ldots, z_n) = f(0,z_2,\ldots, z_n)
+ z_1\int_0^1 \frac {\partial f}{\partial z_1}(z_1t, z_2, \ldots, z_n) dt
$$
so that  
$$
|f(z_1,\ldots, z_n)|_r \leq |f(0,z_2,\ldots,z_n)|_r 
+ r \Big|\frac {\partial f}{\partial z_1}\Big|_r,
$$
where $|g|_r = \max \Big\{ |g(z)| : |z_i| \leq r,\, 1\leq i \leq n\Big\}$ for
$g\in\CC[z_1,\ldots, z_n]$. Since 
$$\max\{|\beta| : \beta\in \MM(\partial f/\partial z_1)\}= k-1$$
and $f(0,z_2,\ldots,z_n)$ is either identically zero or a stable polynomial 
in $n-1$ variables, by the induction assumption we have 
$$
|f(0,z_2,\ldots,z_n)|_r \leq B_1e^{C_1r^2} \quad \mbox{ and } \quad  
\Big|\frac {\partial f}{\partial z_1}\Big|_r \leq B_2e^{C_2r^2}.
$$
Set $C=\max\{C_1, C_2+1\}$ and $B=2\max\{B_1, B_2\}$. 
From the above estimates we get 
$|f|_r \leq  B e^{Cr^2}$, which completes the induction step. 
\end{proof}

\subsection{Proof of Theorem \ref{t-trans}}\label{ss-83}

We can now settle Theorem \ref{t-trans} and thereby -- as explained at the
beginning of this section -- Theorems \ref{multi-infinite-stab} and 
\ref{multi-infinite-hyp} as well. Note that the case $\KK=\RR$ in Theorem \ref{t-trans} follows from the case $\KK=\CC$ since $\HH_{2n}(\CC) \cap \RR[z_1,\ldots, z_n, w_1,\ldots, w_n]= \HH_n(\RR)$. Therefore we assume that $\KK=\CC$. 

Suppose that  $F(z,w)=\sum_{\alpha \in \NN^n}P_\alpha(z)w^\alpha\in 
\HHH_{2n}(\CC)$ and $P_\alpha\in\CC[z_1,\ldots,z_n]$, $\alpha\in\NN^n$. If  
$F_m(z,w) = \sum_{\alpha\in \NN^n} P_{m,\alpha}(z)w^\alpha$ is a sequence of 
polynomials in $\HH_{2n}(\CC)$ converging to $F(z,w)$, uniformly on compacts, 
then $P_{m,\alpha}(z) \rightarrow P_\alpha(z)$ as $m \rightarrow \infty$  
uniformly on compacts for fixed $\alpha \in \NN^n$. Given $\beta\in\NN^n$ 
let $\Lambda$ be the 
linear operator on $\CC[w_1,\ldots, w_n]$ defined by 
$w^\alpha \mapsto (\beta)_\alpha w^\alpha$, $\alpha\in\NN^n$, 
and extend it to a linear operator $\bar{\Lambda}$ on 
$\CC[z_1,\ldots,z_n, w_1,\ldots, w_n]$ by letting 
$\bar{\Lambda}[z^\gamma w^\alpha]=z^\gamma\Lambda[w^\alpha]$, 
$\alpha,\gamma\in\NN^n$. We claim that $\bar{\Lambda}$ 
preserves stability in $2n$ 
variables. Indeed, this amounts 
to saying that its restriction $\bar{\Lambda}_\kappa$ 
to $\CC_\kappa[z_1,\ldots,z_n, w_1,\ldots, w_n]$ preserves stability
for each $\kappa\in\NN^{2n}$. The latter statement is an immediate 
consequence of the fact
that $\Lambda$ preserves stability in $n$ variables (by Lemma \ref{multi}) 
combined with Theorem \ref{multi-finite-stab}  and a straightforward 
computation showing that the symbol $G_{\bar{\Lambda}_\kappa}$ of 
$\bar{\Lambda}_\kappa$ is a stable polynomial in $4n$ variables. 
Now, since $\bar{\Lambda}$ sends all but finitely many monomials 
$w^\alpha$ to $0$ we have that 
$\bar{\Lambda}[F_m(z,w)] \rightarrow \bar{\Lambda}[F(z,w)]$ as 
$m \rightarrow \infty$, uniformly on compacts, which gives 
$\bar{\Lambda}[F(z,w)] \in \HH_{2n}(\CC)\cup\{0\}$. This proves the necessity
part of Theorem \ref{t-trans}.

To prove the converse assume that 
$\sum_{\alpha \leq \beta}(\beta)_{\alpha}P_\alpha(z)w^\alpha\in 
\HH_{2n}(\CC)\cup\{0\}$ for all $\beta\in\NN^n$. Let $m\in\NN$, set
$\beta_m= (m,\ldots,m)$ and note that 
$$
F_m(z,w):=\sum_{\alpha \leq \beta_m} J(\alpha,\beta_m)P_\alpha(z)w^\alpha
=\sum_{\alpha \leq \beta_m}(\beta_m)_{\alpha}
P_\alpha(z)\left(\frac{w}{m}\right)^\alpha\in \HH_{2n}(\CC)\cup\{0\}.
$$
We want to show that there is a positive integer $M$ such that  
for all $r>0$ there is a constant $K_r$ (depending only on $r$) 
such that for all $m>M$ we have 
$|F_m(z,w)| \leq K_r$ whenever $|w_j| \leq r$ and $|z_j| \leq r$ for 
$1\le j\le n$. This will prove the sufficiency part of Theorem \ref{t-trans} 
since $\{F_m(z,w)\}_{m\in\NN}$ is then a normal family  
whose convergent subsequences converge to $F(z,w)$ (by the fact that 
$\lim_{m\rightarrow \infty}J(\alpha,\beta_m) = 1$ for all 
$\alpha \in \NN^n$). Now, there exists a positive integer $N$ such 
that the set $\MM(F_m)$ of minimal elements of $\supp(F_m)$ does not change 
for $m \geq N$ (this is easily seen from the fact that all lower degree terms
are multiplied by eventually non-zero numbers $J(\alpha,\beta_m)$). 
Since $\lim_{m\rightarrow \infty}J(\alpha,\beta_m) = 1$ it follows from 
Theorem~\ref{for_trans} that there is a positive integer $M$ and 
constants $B$ and $C$ such that if $|w_j| \leq r$ and $|z_j| \leq r$, 
$1\le j\le n$, then 
$$
|F_m(z,w)| \leq Be^{Cr^2},  \quad m >M,
$$
which completes the proof of the theorem.


\section{Circular Domains and Lee-Yang Polynomials}\label{s-CDLY}

We will now solve Problems \ref{prob1}--\ref{prob2} whenever $\Omega$ is a 
product of arbitrary open circular domains. Using this 
we then characterize all linear transformations preserving the class of 
Lee-Yang polynomials defined with respect to any such set $\Omega$.

\subsection{Products of Open Circular Domains}\label{OCCD}

Recall that a {\em M\"obius transformation} is a bijective 
conformal map of the extended complex plane $\widehat{\CC}$ given by 
\begin{equation}\label{mobius}
\phi(\zeta)= \frac {a\zeta+b}{c\zeta+d}, \quad a,b,c,d \in \CC, ad-bc =1. 
\end{equation}
Note that one usually has the weaker requirement $ad-bc \neq 0$ but 
this is equivalent to \eqref{mobius} which proves to be 
more convenient.
An {\em open circular domain} is the 
image (in $\CC$) of the upper half-plane $H$ under a 
M\"obius transformation, i.e., an open disk, the complement of a closed disk 
or an open affine half-plane. 

As usual, the {\em support}  of a polynomial 
$f(z)=\sum_{\alpha \in \NN^n} a(\alpha) z^\alpha \in \CC[z_1,\ldots, z_n]$, 
i.e., the set $\{\alpha \in \NN^n : a(\alpha) \neq 0\}$, is 
denoted by $\supp(f)$. 
We say that $f$ is of degree $\kappa=(\kappa_1,\ldots,\kappa_n) \in \NN^n$ if 
$\deg_{z_i}(f)=\kappa_i$, $i\in [n]$, where as before $[n]=\{1,\ldots,n\}$. 

To deal with discs and exterior of discs we need some auxiliary results.

\begin{lemma}\label{maxsupport}
Let $\{C_i\}_{i=1}^n$ be a family of circular domains, 
$f \in \CC[z_1,\ldots, z_n]$ be of degree $\kappa\in\NN^n$, and 
$J \subseteq [n]$ a (possibly empty) set such that $C_j$ is the exterior of 
a disk whenever $j \in J$. Denote by $g$ 
the polynomial in the variables $z_j$, $j \in J$, obtained by setting 
$z_i = c_i \in C_i$ arbitrarily for $i \notin J$. If $f$ is 
$C_1\times\cdots\times C_n$-stable then 
$\supp(g)$ has a unique maximal element $\gamma$ with respect to the 
standard partial order on $\NN^J$. Moreover, 
$\gamma$ is the same for all choices of $c_i \in C_i$, $i \notin J$. 
\end{lemma}

\begin{proof} 
Let us first prove the lemma in the case when $J=[n]$. 
We may assume that $C_j$ is the exterior of the 
 closed unit disk $\overline{\DD}$ for all $j \in [n]$. Suppose that there is 
no unique maximal element of $\supp(f)$. For fixed  
$\lambda=(\lambda_1,\ldots,\lambda_n) \in \RR_+^n$ and 
$\beta=(\beta_1,\ldots,\beta_n)\in\NN^n$  consider the univariate polynomial 
$$
F(\lambda,\beta ;t):= K f(\lambda_1t^{\beta_1}, \ldots, \lambda_nt^{\beta_n}),
$$
where $K=K(\lambda,\beta) \in \CC$ is the appropriate constant making $F$ 
monic.  
Let $\mathcal{A}(\beta)$ be the closure (with respect to uniform convergence 
on compacts) of the set 
$
\{ F(\lambda,\beta;t) : \lambda \in (1, \infty)^n \}. 
$
By Hurwitz' theorem (Theorem \ref{mult-hur}) 
all elements of $\mathcal{A}(\beta)$ are 
$\CC\setminus\overline{\DD}$-stable. However, since there is no unique 
maximal element of $\supp(f)$ there exists $\beta\in\NN^n$ such that 
$\mathcal{A}(\beta)$ contains polynomials of different degrees. 
Hence, as we vary $\lambda$ at least one zero must escape to infinity. 
This is impossible since all the zeros of $F(\lambda,\beta;t)$ have modulus 
at most one. 

If $J \neq [n]$ and the maximal elements are different for different choices 
of $c_i$ then as above some zero of a specialization of $f$ would escape to 
infinity contradicting the boundedness of the zeros.   
\end{proof}
When dealing with non-convex circular domains one has to be careful with 
degrees. We are therefore led to consider the following classes of 
polynomials. 

\begin{notation}\label{not-nv}
If $\{C_i\}_{i=1}^n$ is a family of circular domains and 
$\kappa=(\kappa_1,\ldots,\kappa_n) \in \NN^n$ we let 
$\NV_\kappa(C_1,\ldots,C_n)$ denote the set of 
$C_1 \times \cdots \times C_n$-stable polynomials in  
 $\CC_\kappa[z_1,\ldots,z_n]$ that have degree $\kappa_j$ in 
$z_j$ whenever $C_j$ is non-convex. (If all $C_j$ are convex then 
$\NV_\kappa(C_1,\ldots,C_n)$ is just the set of all 
$C_1 \times \cdots \times C_n$-stable polynomials in  
$\CC_\kappa[z_1,\ldots,z_n]$.) Let also
$$
\NV(C_1,\ldots,C_n)=\bigcup_{\kappa\in\NN^n}\NV_\kappa(C_1,\ldots,C_n). 
$$
\end{notation}

\begin{lemma}\label{translate}
Suppose that $C_1, \ldots, C_n, D_1,\ldots, D_n$ are open circular domains and 
$\kappa=(\kappa_1,\ldots,\kappa_n) \in \NN^n$. Then there are M\"obius 
transformations 
\begin{equation}\label{mob-n}
\zeta\mapsto \phi_i(\zeta)=\frac{a_i\zeta+b_i}{c_i\zeta+d_i},\quad i\in [n],
\end{equation}
as in \eqref{mobius} such that the (invertible) 
linear transformation  
$\Phi_\kappa:  \CC_\kappa[z_1,\ldots, z_n] \rightarrow 
\CC_\kappa[z_1, \ldots, z_n]$ defined by 
\begin{equation}\label{Phik}
\Phi_\kappa(f)(z_1,\ldots,z_n)
=(c_1z_1+d_1)^{\kappa_1}\cdots (c_nz_n+d_n)^{\kappa_n}
f(\phi_1(z_1),\ldots,\phi_n(z_n))
\end{equation}
restricts to a bijection between $\NV_\kappa(C_1,\ldots,C_n)$ and 
$\NV_\kappa(D_1,\ldots,D_n)$. 
\end{lemma}

\begin{proof}
The relation between tuples of open circular domains defined by relating 
tuples for which there is such 
a linear transformation $\Phi_\kappa$ is an equivalence relation. Therefore it 
suffices to assume that 
$C_j=D_j$ for $j\geq 2$ and that $D_1=H\,(=\bH_0)$. Thus we want to find 
$\Phi_\kappa$ of the form 
$$
\Phi_\kappa(f)= 
(cz_1+d)^{\kappa_1} f\!\left(\frac{az_1+b}{cz_1+d},z_2,\ldots,z_n\right). 
$$
If $C_1$ is convex we may chose $\phi(z)= (az+b)/(cz+d)$ as in \eqref{mobius} 
so that $\phi(H)=C_1$ and  $cz+d \neq 0$ for all $z \in H$. If $C_1$ is 
non-convex we may chose $\phi$ so that $\phi(H\setminus\{\eta\})=C_1$ for
 some specific point $\eta \in H$, namely $\eta=-d/c$. 

Now let $f \in \NV_\kappa(C_1,\ldots,C_n)$ and  $\xi=(\xi_1,\ldots, \xi_n) \in D_1\times \cdots \times D_n$. If $c\xi_1+d \neq 0$ then $\phi(\xi_1) \in C_1$ and $\Phi_\kappa(f)(\xi)\neq 0$ by construction. If $c\xi_1+d = 0$ then $C_1$ is non-convex so by  Lemma \ref{maxsupport} the polynomial 
$z_1 \mapsto f(z_1,\xi_2, \xi_3, \ldots, \xi_n)$ may be written as 
$
\sum_{k=0}^{\kappa_1}\alpha_k z_1^k
$
with $\alpha_{\kappa_1}\neq 0$. But then 
$$
\Phi_\kappa(f)(\xi_1,\ldots, \xi_n)
= \alpha_{\kappa_1}(ad-bc)^{\kappa_1}/c^{\kappa_1} 
\neq 0. 
$$
Hence $\Phi_\kappa(f) \in \NV_\kappa(D_1,\ldots,D_n)$.

If $g \in \NV_\kappa(D_1,\ldots,D_n)$ consider 
$$
\Phi^{-1}_\kappa(g)= 
(-cz_1+a)^{\kappa_1} g\!\left(\frac{dz_1-b}{-cz_1+a},z_2,\ldots,z_n\right),
$$
where $\phi(z)=(az+b)/(cz+d)$ is chosen as above. It follows that $-cz+d \neq 0$ for all $z \in C_1$ so it only remains to prove that the degree of $\Phi^{-1}_\kappa(g)$ in $z_1$ is $\kappa_1$ when $C_1$ is non-convex. To do this fix $\xi_j \in D_j$, $2\le j\le n,$ and consider the $H$-stable polynomial 
$g(z_1,\xi_2,\ldots,\xi_n)= \sum_{k=0}^{\kappa_1}a_kz_1^k$. The coefficient in front of $z_1^{\kappa_1}$ in 
$
\Phi^{-1}_\kappa(g)(z_1,\xi_2,\ldots,\xi_n)
$
is $(-c)^{\kappa_1}\sum_{k=0}^{\kappa_1}a_k(-d/c)^k$ which is non-zero since as we previously noted one has $-d/c \in H$. 
\end{proof}

The following theorem answers a more precise version of 
Problem \ref{prob1} for $\KK=\CC$ and 
$\Omega=C_1\times\cdots\times C_n$, where the $C_i$ are arbitrary open 
circular domains.

\begin{theorem}\label{th-prod-circ}
Let $\kappa \in \NN^n$ and $T : \CC_\kappa[z_1,\ldots, z_n] \rightarrow 
\CC[z_1,\ldots, z_n]$ be a linear operator. Let further 
$C_i$, $i\in [n]$, be open 
circular domains given by $C_i=\phi_i^{-1}(H)$, $i\in [n]$, 
where 
$\phi_i$, $i\in [n]$, are M\"obius transformations as in \eqref{mob-n} 
such that the corresponding linear transformation $\Phi_\kappa$ defined in 
\eqref{Phik} restricts to 
a bijection between $\NV_\kappa(H,\ldots,H)$ and  $\NV_\kappa(C_1,\ldots,C_n)$ 
 (cf.~Lemma \ref{translate}). Then
$$
T:\NV_\kappa(C_1,\ldots,C_n)\to \NV(C_1,\ldots,C_n)\cup \{0\}
$$
if and only if either
\begin{itemize}
\item[(a)] $T$ has range of dimension at most one and is of the form 
$$
T(f) = \alpha(f)P,
$$
where $\alpha$ is a linear functional on $\CC_\kappa[z_1,\ldots, z_n]$ and 
$P$ is a $C_1\times \cdots \times C_n$-stable polynomial, or 
\item[(b)] the polynomial in $2n$ variables $z_1,\ldots,z_n,w_1,\ldots,w_n$ 
given by
$$
T\left[\prod_{i=1}^{n}
\big((a_iz_i+b_i)(c_iw_i+d_i)+(a_iw_i+b_i)(c_iz_i+d_i)\big)^{\kappa_i}\right]
$$ 
is $C_1\times \cdots \times C_n\times C_1\times \cdots \times C_n$-stable. 
\end{itemize}
\end{theorem}

\begin{remark}\label{unit-disk}
By analogy with the definitions made in \S \ref{new-ss-not} we may call the
 polynomial in 
Theorem \ref{th-prod-circ} (b) the {\em algebraic symbol of} $T$ {\em with 
respect to} $C_1\times \cdots \times C_n$. Note that in the case of the open 
unit disk $C_i=\DD$, $i\in [n]$, or the exterior of the closed unit disk 
$C_i=\CC\setminus \overline{\DD}$, $i\in [n]$, this 
symbol becomes a constant multiple of 
$T[(1+zw)^\kappa]$ while in the case of an open half-plane bordering on 
the origin it is just a constant multiple of $G_T(z,w)=T[(z+w)^\kappa]$. 
For the open right
half-plane $\bH_{\frac{\pi}{2}}$ it is often more convenient 
(but equivalent) to choose the symbol $T[(1+zw)^\kappa]$. 
\end{remark}

\begin{proof}[Proof of Theorem~\ref{th-prod-circ}] 
Let $T: \CC_\kappa[z_1,\ldots, z_n] \rightarrow \CC[z_1,\ldots, z_n]$ 
be a linear operator and $\gamma=(\gamma_1,\ldots,\gamma_n)\in\NN^n$ 
be such that 
$T(\CC_\kappa[z_1,\ldots, z_n])\subseteq \CC_\gamma[z_1,\ldots, z_n]$. 
The theorem is trivial if the range of $T$ has dimension at most one so we may 
assume that this is not the case. With $\phi_i$, $i\in [n]$, and $\Phi_\kappa$ 
as in the statement of the theorem let 
$\Phi_\gamma:\CC_\gamma[z_1,\ldots, z_n] \rightarrow 
\CC_\gamma[z_1,\ldots, z_n]$ be the linear transformation defined as in 
\eqref{Phik} only with $\gamma=(\gamma_1,\ldots,\gamma_n)$ instead of 
$\kappa=(\kappa_1,\ldots,\kappa_n)$. Set $T':=\Phi^{-1}_\gamma T \Phi_\kappa$. 
By Theorem~\ref{multi-finite-stab} and 
Lemma~\ref{translate} we have the following chain of equivalences:
\begin{equation*}
\begin{split}
\text{$T(\NV_\kappa(C_1,\ldots,C_n))\subset \NV(C_1,\ldots,C_n)\cup \{0\}$}
&\Longleftrightarrow\\
\text{ $T'$ preserves ($H$-)stability}&\Longleftrightarrow\\ 
\text{ $G(z,w):=T'[(z+w)^\kappa]$ is ($H$-)stable }&\Longleftrightarrow\\ 
\text{ $\Phi_{\gamma \oplus \kappa} (G)$ is 
$C_1\times \cdots \times C_n\times C_1\times \cdots \times C_n$-stable},&
\end{split}
\end{equation*} 
where $\Phi_{\gamma \oplus \kappa}:
\CC_{\gamma \oplus \kappa}[z_1,\ldots,z_n,w_1,\ldots,w_n]\to
\CC_{\gamma \oplus \kappa}[z_1,\ldots,z_n,w_1,\ldots,w_n]$ is given by
\begin{multline*}
\Phi_{\gamma \oplus \kappa}(f)(z_1,\ldots,z_n,w_1,\ldots,w_n)\\
=\prod_{i=1}^{n}(c_iz_i+d_i)^{\gamma_i}
\prod_{i=1}^{n}(c_iw_i+d_i)^{\kappa_i}
f(\phi_1(z_1),\ldots,\phi_n(z_n),\phi_1(w_1),\ldots,\phi_n(w_n)),
\end{multline*}
see \eqref{Phik}. Now $\Phi_{\gamma \oplus \kappa} (G)$ is precisely the 
polynomial in 
Theorem~\ref{th-prod-circ} (b). Indeed, letting subscript $z$ and subscript 
$w$ indicate that the 
corresponding operator acts on the $z$ and $w$ variables, respectively, we have
\begin{equation*}
\begin{split}
\Phi_{\gamma \oplus \kappa}(G)(z,w)&
=\big(\Phi_{\kappa,w} \circ \Phi_{\gamma,z} 
\circ \Phi^{-1}_{\gamma,z} \circ T_z \circ \Phi_{\kappa,z}\big) 
\Big[(z+w)^\kappa\Big]\\
&= T\Big[ \Phi_{\kappa \oplus \kappa}\Big( (z+w)^\kappa \Big) \Big],
\end{split}
\end{equation*}
which completes the proof of the theorem.
\end{proof}

Note that if all $C_i$ are open {\em convex} circular domains then  
Theorem \ref{th-prod-circ} answers Problem \ref{prob1} (for such domains) 
precisely as stated in the introduction, and when non-convex circular
 domains are involved 
it answers a more precise version of the problem.

\subsection{Linear Preservers of Lee-Yang Polynomials}\label{s-LY} 
Given an open circular domain $C\subset \CC$ let $C^r$ be the interior 
of the complement of 
$C$, that is, $C^r=\CC\setminus \overline{C}$. 

\begin{definition}\label{d-ly}
Let $\kappa=(\kappa_1,\ldots,\kappa_n)\in\NN^n$ and  
$\{C_i\}_{i=1}^n$ be a family of open circular domains. The set of 
{\em $\kappa$-Lee-Yang polynomials with respect to} 
$\{C_i\}_{i=1}^{n}$ is 
$$
\LY_\kappa(C_1,\ldots,C_n)=\NV_\kappa(C_1,\ldots,C_n)\cap 
\NV_\kappa(C_1^r,\ldots,C_n^r).
$$ 
A {\em Lee-Yang polynomial  
with respect to} $\{C_i\}_{i=1}^{n}$ is an element of the set
$$
\LY(C_1,\ldots,C_n):=\bigcup_{\kappa\in\NN^n}
\LY_\kappa(C_1,\ldots,C_n).
$$
\end{definition} 

This definition is motivated by the original 
Lee-Yang theorem \cite{LY}  and its various 
versions such as the so-called ``circle theorem'' 
(\cite[Theorem 8.4]{BB-II}, \cite{hink}, \cite{ruelle5}). For instance,  
the class of 
``complex Lee-Yang polynomials'' related to the latter which 
was introduced and studied in \cite{beau} corresponds to the case when 
$C_i=\DD$, $i\in [n]$, in Definition \ref{d-ly}. Note also that 
by Corollary \ref{cor-HB} the Lee-Yang polynomials with respect 
to $\{H_i\}_{i=1}^n$, where $H_i=H$, $i\in [n]$, are precisely the polynomials 
that are non-zero complex constant multiples of real stable polynomials. 
 
We say that a linear operator $T : \CC_\kappa[z_1,\ldots, z_n] \rightarrow 
\CC[z_1,\ldots, z_n]$ {\em preserves the $\kappa$-Lee-Yang property with 
respect to} $\{C_i\}_{i=1}^{n}$ 
if $T(f)\in\LY(C_1,\ldots,C_n)\cup \{0\}$ whenever 
$f\in \LY_\kappa(C_1,\ldots,C_n)$. From our symbol theorems we 
can deduce a characterization of all such linear operators. We only display 
the case when $T$ is {\em non-degenerate} in the sense that its image has 
dimension larger than 2. 

\begin{theorem}\label{th-ly-pres}
Let $\kappa \in \NN^n$, $T : \CC_\kappa[z_1,\ldots, z_n] \rightarrow 
\CC[z_1,\ldots, z_n]$ be a non-degenerate linear operator and 
$C_i$, $i\in [n]$, be open 
circular domains given by $C_i=\phi_i^{-1}(H)$, $i\in [n]$, 
where $\phi_i$, $i\in [n]$, are M\"obius transformations as in 
Theorem \ref{th-prod-circ}. Then $T$ preserves the $\kappa$-Lee-Yang property 
with respect to $\{C_i\}_{i=1}^{n}$ if and only if either
\begin{itemize}
\item[(a)] the polynomial in $2n$ variables $z_1,\ldots,z_n,w_1,\ldots,w_n$ 
given by
$$
T\left[\prod_{i=1}^{n}
\big((a_iz_i+b_i)(c_iw_i+d_i)+(a_iw_i+b_i)(c_iz_i+d_i)\big)^{\kappa_i}\right]
$$ 
is $C_1\times \cdots \times C_n\times C_1\times \cdots \times C_n$-stable and 
$C_1^r\times \cdots \times C_n^r\times C_1^r\times \cdots \times 
C_n^r$-stable, or
\item[(b)] the polynomial in $2n$ variables $z_1,\ldots,z_n,w_1,\ldots,w_n$ 
given by
$$
T\left[\prod_{i=1}^{n}
\big((a_iz_i+b_i)(c_iw_i+d_i)-(a_iw_i+b_i)(c_iz_i+d_i)\big)^{\kappa_i}\right]
$$ 
is $C_1\times \cdots \times C_n\times C_1\times \cdots \times C_n$-stable and 
$C_1^r\times \cdots \times C_n^r\times C_1^r\times \cdots \times 
C_n^r$-stable.
\end{itemize}
\end{theorem}

\begin{proof}
The proof of the corresponding theorem in the case $n=1$ from  \cite{BBS1},
namely \cite[Theorem 3]{BBS1}, extends naturally to the above setting. 
\end{proof}

\section{Related Results}\label{ss-95}





A multivariate polynomial is called {\em strictly stable} if it is 
non-vanishing whenever all variables lie in 
$\overline{H}=\{z\in\CC:\Im(z)\ge 0\}$. The arguments in \S \ref{s-6} 
carry over to the closed upper half-plane $\overline{H}$ and yield the 
following {\em sufficient} condition for strict stability preserving.

\begin{theorem}\label{th-strict}
Let $\kappa \in \NN^n$ and $T : \CC_\kappa[z_1,\ldots, z_n] \rightarrow 
\CC[z_1,\ldots, z_n]$ be a linear operator. If $G_T(z,w)$ is a strictly stable
polynomial then $T$ preserves strict stability. 
\end{theorem}

One can also check that the sufficiency part of Theorem \ref{th-prod-circ} 
carries over to closed circular domains and thus yields a (sufficient) 
condition for linear operators $T : \CC_\kappa[z_1,\ldots, z_n] \rightarrow 
\CC[z_1,\ldots, z_n]$ to preserve 
$\overline{C}_1\times\cdots\times\overline{C}_n$-stability, where
$C_i$, $i\in [n]$, are arbitrary (open) circular domains. For simplicity we 
only state
here a special case that extends Theorems \ref{th-strict}.

\begin{theorem}\label{th-circ-strict}
Let $\kappa \in \NN^n$, $T : \CC_\kappa[z_1,\ldots, z_n] \rightarrow 
\CC[z_1,\ldots, z_n]$ a linear operator and $\overline{C}$ a convex closed 
circular domain given by $\overline{C}=\phi^{-1}(\overline{H})$, 
where $\phi$ is a M\"obius transformation as in \eqref{mobius}. If 
$$
T\left[\big((az+b)(cw+d)+(aw+b)(cz+d) \big)^\kappa\right]
$$ 
is $\overline{C}$-stable then $T$ preserves $\overline{C}$-stability.
\end{theorem}

As noted in Remark \ref{unit-disk}, for the the closed unit disk 
$\overline{C}=\overline{\DD}$ the symbol used in 
Theorem \ref{th-circ-strict} is just a constant multiple of 
$T[(1+zw)^{\kappa}]$ while for a closed half-plane
bordering on the origin it is a constant multiple of 
$G_T(z,w)=T[(z+w)^\kappa]$. Note though that the conditions in 
Theorems \ref{th-strict}, \ref{th-circ-strict} are not necessary. For 
instance, the identity operator obviously preserves strict stability but its
(algebraic) symbol $(z+w)^\kappa$ is not strictly stable, cf.~(i) in 
\S \ref{s-ide}.

%





\section{Further Directions}\label{s-ide}

To conclude we mention some of the most appealing cases where 
Problems \ref{prob1}--\ref{prob2} remain open:

\begin{itemize}
\item[(i)] $\Omega$ is a closed circular domain, to begin 
with $\Omega=\overline{H}$ (cf.~Theorem \ref{th-strict});
\item[(ii)] $\Omega$ is a strip, e.g., $\Omega=\{z\in \CC:|\Im(z)|<a\}$, 
$a>0$;
\item[(iii)] $\Omega$ is a sector, e.g., 
$\Omega=\left\{z\in \CC:\frac{\pi}{2}-\alpha<\arg(z)<\frac{\pi}{2}
+\alpha\right\}$, $\alpha\in \left(0,\frac{\pi}{2}\right)$;
\item[(iv)] $\Omega$ is a closed or open strip or (double) sector.
\end{itemize}

Finally, from a complex analytic perspective and its many 
potential applications
\cite{BBCV} it would be quite useful to establish 
analogs of these results for transcendental entire functions in one or 
several variables. 

\section*{Acknowledgments}

We wish to thank the Isaac Newton Institute for Mathematical Sciences, 
University of Cambridge, for generous support during the programme  
``Combinatorics and Statistical Mechanics'' (January--June 2008), 
where this work, \cite{BB-II} and \cite{BBL} 
were presented in a series of lectures. Special thanks are due 
to the organizers and the participants in this programme for 
numerous stimulating discussions. 
We would also like to thank Alex Scott, 
B\'ela Bollob\'as and Richard Kenyon for the opportunities to announce 
these results in June--July 2008 through talks given at the 
Mathematical Institute, University of Oxford, the Centre for Mathematical 
Sciences, University of Cambridge, and the programme ``Recent Progress in 
Two-Dimensional Statistical Mechanics'', Banff International 
Research Station, respectively.

\end{document}